\documentclass[a4paper,11pt,fleqn]{article}

\usepackage[margin=1in]{geometry}
\usepackage[T1]{fontenc}
\usepackage{lmodern}
\usepackage{microtype}

\usepackage{amsmath,amssymb,amsfonts}
\usepackage{amsthm}
\usepackage{bm}
\usepackage{mathtools}
\usepackage{enumitem}
\usepackage{graphicx}
\usepackage{booktabs}
\usepackage{multirow}
\usepackage[numbers]{natbib}
\usepackage{xcolor}
\usepackage[hidelinks]{hyperref}

\newif\ifshowchanges
\showchangesfalse

\ifshowchanges
  \newcommand{\rev}[1]{{\color{blue}#1}}
  \newenvironment{revision}{\begingroup\color{blue}}{\endgroup}
  \newcommand{\newrev}[1]{{\color{magenta}#1}}
  \newenvironment{newrevision}{\begingroup\color{magenta}}{\endgroup}
  \newcommand{\finalrev}[1]{{\color{green!50!black}#1}}
  \newenvironment{finalrevision}
    {\begingroup\color{green!50!black}}
    {\endgroup}
\else
  \newcommand{\rev}[1]{#1}
  \newenvironment{revision}{}{}
  \newcommand{\newrev}[1]{#1}
  \newenvironment{newrevision}{}{}
  \newcommand{\finalrev}[1]{#1}
  \newenvironment{finalrevision}{}{}
\fi

\newcommand{\R}{\mathbb{R}}
\newcommand{\C}{\mathbb{C}}
\newcommand{\Z}{\mathbb{Z}}
\newcommand{\indicator}[1]{\mathbf{1}_{#1}}
\newcommand{\Ph}{P^{h}}
\newcommand{\Qh}{Q^{h}}
\newcommand{\Gh}{G^{h}}
\newcommand{\Lh}{L^{h}}
\newcommand{\Omegain}{\Omega_{\mathrm{in}}}
\newcommand{\Gammab}{\Gamma_{\mathrm{b}}}
\DeclarePairedDelimiter{\abs}{\lvert}{\rvert}

\newcommand{\supp}{\operatorname{supp}}

\newtheorem{theorem}{Theorem}[section]
\newtheorem{proposition}[theorem]{Proposition}

\theoremstyle{definition}
\newtheorem{definition}[theorem]{Definition}

\theoremstyle{remark}
\newtheorem{remark}[theorem]{Remark}

\title{\finalrev{Infinite-lattice discrete Calder\'on projection
via the lattice Green's function for active noise shielding and
confinement}}

\author{%
Qing Xia\thanks{Corresponding author:
\href{mailto:qxia@kean.edu}{qxia@kean.edu}}\\[0.4em]
\small Department of Mathematics, Wenzhou Kean University,\\
\small Wenzhou 325060, China\\[0.2em]
}

\date{}

\begin{document}

\maketitle

\begin{abstract}
\finalrev{We construct an infinite-lattice discrete Calder\'on projection for the
Helmholtz equation by convolution with the lattice Green's function
(LGF), and apply it to active noise shielding and confinement on
Cartesian grids with arbitrary geometry. The LGF fixes the outgoing
radiation condition and removes the geometry-dependent auxiliary
Helmholtz problem and artificial outer boundary from the projection and
control synthesis; its finite numerical tabulation depends only on
$(h,k)$ and is reusable across geometries. We prove idempotence,
characterize the range as the trace space of interior
lattice-Helmholtz solutions, and establish range equivalence with a
well-posed Tsynkov-type projection. The two projectors coincide as
operators when the auxiliary problem reproduces the exact lattice
radiation condition. A capacity-matrix realization yields closed-form
shielding and confinement densities supported on the exterior and
interior sublayers of a single lattice boundary strip, respectively.
For pure-noise shielding, exterior-sublayer measurements suffice under
explicit invertibility assumptions; preservation of an unknown wanted
interior field requires the full strip trace. Experiments on circular,
L-shaped, and star-shaped regions verify machine-precision cancellation
for LGF-consistent sources and near-second-order convergence for
analytic plane waves and point sources. Conditioning and measurement
noise tests quantify the configuration dependence of the reconstruction.}
\end{abstract}

\rev{one-sided measurement}

\medskip

\section{Introduction}\label{sec:intro}

\begin{finalrevision}
Active noise control (ANC) suppresses acoustic disturbances by radiating
a secondary field that cancels the adverse field in a prescribed region.
Since Lueg's 1936 patent~\cite{Lueg1936}, ANC has developed from
single-point adaptive filtering to spatial wave-domain control and
multi-zone sound reproduction~\cite{KuoMorgan1999,Elliott2001,
ZhangAbhayapalaZhangSamarasingheJiang2018,
BetlehemZhangPolettiAbhayapala2015}.  The central challenge for extended
regions is to construct a physically outgoing control field that acts
through boundary data while remaining insensitive to the unknown source
configuration.
\end{finalrevision}

A complementary mathematical approach was initiated by
Lon\v{c}ari\'c, Ryaben'kii, and Tsynkov~\cite{LoncaricRyabenkiiTsynkov2001},
who reformulated ANC through Calder\'on boundary projections and
generalized surface potentials.  In this framework, the total
time-harmonic field decomposes into outgoing and ingoing components
relative to the protected region $\Omegain$, and the controls are
obtained from the trace of the total field alone on
$\partial\Omegain$, without prior knowledge of the noise sources, the
wanted interior sources, or the medium parameters outside a
neighbourhood of $\partial\Omegain$.  The discrete realization through
the method of difference potentials was developed in
\cite{Tsynkov2003,Ryabenkii2002}, extended to composite domains in
\cite{PetersenTsynkov2007,RyabenkiiTsynkovUtyuzhnikov2007}, optimised
in \cite{LoncaricTsynkov2003,LoncaricTsynkov2005}, validated
experimentally on ducts in \cite{LimEtAl2009AIAA,LimEtAl2011JASA},
and carried to three dimensions in
\cite{LimEtAl2014JASA,NtumyUtyuzhnikov2014,NtumyUtyuzhnikov2015}.
The nonlocal active sound control (NASC) scheme of Utyuzhnikov,
Hu, and Zhou~\cite{Utyuzhnikov2017,ZhouUtyuzhnikov2020,
HuUtyuzhnikov2022,HuUtyuzhnikov2024} places two concentric Huygens
surfaces around the protected region and is the most advanced
practical realization of this programme to date.

In all these works, the discrete Calder\'on projection
$\Ph_{\gamma}$ is constructed on a finite auxiliary domain
$\Omega_{0}\supset\Omegain$ by inverting the Helmholtz operator on
$\Omega_{0}$ with a chosen outer boundary condition.  This paper
presents an alternative construction that \emph{replaces the finite
auxiliary inverse by convolution with the Helmholtz lattice Green's
function} (LGF) on the infinite lattice $h\mathbb{Z}^{d}$.
\begin{finalrevision}
The resulting projection is infinite-lattice: no artificial outer
boundary enters the discrete wave representation.  Its implementation
uses the finite LGF tabulation of~\cite{WangXia2025}, which depends only
on $(h,k)$ and is independent of the protected geometry.
\end{finalrevision}
The framework handles both \emph{shielding} (cancelling exterior noise
inside a protected region) and its dual, \emph{confinement}
(cancelling interior noise outside), through complementary
projections.  Our main contributions are as follows.

\begin{enumerate}[leftmargin=1.8em]
\item \textbf{Infinite-lattice projection and trace space.}
\finalrev{We define $\Ph_{\gamma}$ by LGF convolution, prove that it is a
projection, and characterize its range as the trace space of interior
lattice-Helmholtz solutions.  This range agrees with that of any
well-posed Tsynkov-type construction, while the LGF selects the outgoing
complement directly through the lattice radiation condition.}

\item \textbf{Capacity-matrix control formulas.}
\finalrev{A single dense solve on the lattice boundary strip separates
the two Calder\'on components.  The exterior-sublayer density cancels
adverse sound inside the protected region while preserving wanted
interior sound; the interior-sublayer density gives the dual confinement
control and preserves an ambient exterior field.}

\item \textbf{Conditional one-sided sensing.}
\finalrev{For pure-noise shielding, measurements on $\gamma^{-}$ determine
the missing interior trace when the exterior single-layer matrix and the
interior Dirichlet map are invertible.  An unknown wanted interior field
requires the full two-sided trace or prior subtraction of its known
contribution.}

\item \textbf{Geometry-independent discretization and validation.}
\finalrev{Level-set classification extracts the boundary strip without
body fitting.  Two-dimensional tests on circular, reentrant, and
non-convex regions verify the exact discrete identities, near-second-order
accuracy for analytic fields, and the effects of conditioning and
measurement noise.  The theory itself applies in $d\in\{2,3\}$.}
\end{enumerate}


The use of lattice Green's functions for Helmholtz problems has a
substantial history.  Bamberger, Guillot, and Joly~\cite{BambergerGuillotJoly1988}
pioneered the numerical treatment of diffraction by a uniform grid
using discrete Fourier techniques.  Martin~\cite{Martin2006} developed
a discrete scattering theory and derived the Green's function for a
square lattice via contour integration.  Bhat and Osting~\cite{BhatOsting2010}
extended this to diffraction problems on the two-dimensional square
lattice with detailed asymptotic analysis.  Poblet-Puig, Valyaev, and
Shanin~\cite{PobletPuigValyaevShanin2015} identified and suppressed
spurious frequencies in discrete scattering through boundary algebraic
and combined field equations, and Poblet-Puig and Shanin~\cite{PobletPuigShanin2018}
applied these techniques to acoustic radiation problems.

The Helmholtz LGF used here is computed by the method of Wang and
Xia~\cite{WangXia2025}, which combines a discrete sine transform with a
Hankel-function boundary correction.  \begin{finalrevision}
The new contribution is the use of this outgoing infinite-lattice kernel
to realize the discrete Calder\'on projection and its active controls
entirely through boundary-strip algebra.  The theory is formulated for
$d\in\{2,3\}$; the numerical study is restricted to two dimensions.
\end{finalrevision}


\finalrev{Section~\ref{sec:prelim} introduces the lattice geometry and
LGF.  Sections~\ref{sec:projection}--\ref{sec:onesided} develop the
projection, control formulas, and one-sided reconstruction;
Sections~\ref{sec:numerics} and~\ref{sec:algorithm} present the numerical
results and computational realization.}

\section{Preliminaries: lattice, Helmholtz LGF, boundary strip}\label{sec:prelim}

\begin{figure}[t]
\centering
\includegraphics[width=0.4\textwidth]{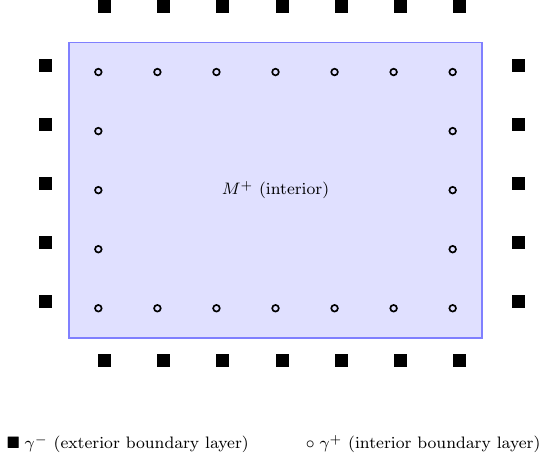}
\caption{Cartesian lattice $\Lambda_h=h\mathbb{Z}^{2}$ with a rectangular
protected region $M^{+}$ (blue).  The boundary strip $\gamma=\gamma^{+}\cup\gamma^{-}$
is the set of lattice nodes whose 5-point stencil straddles the boundary:
$\gamma^{+}$ (orange) lies just inside $M^{+}$, $\gamma^{-}$ (red) lies just
outside.  The shielding control $g^{h}$ is supported on the single exterior
layer $\gamma^{-}$.}
\label{fig:lattice}
\end{figure}

\paragraph{Lattice} Fix $h>0$ and let $\Lambda_{h}=h\Z^{d}$,
${d\in\{2,3\}}$. Grid functions are maps $v^{h}:\Lambda_{h}\to\C$ with
appropriate summability. We denote by $\Lh$ the standard second-order
finite-difference Helmholtz operator at wavenumber $k$,
\begin{equation}
  (\Lh v^{h})_{\bm{m}}
   \;=\;-\frac{1}{h^{2}}\sum_{\abs{\bm{e}}_{1}=1}
        \bigl(v^{h}_{\bm{m}+\bm{e}}-v^{h}_{\bm{m}}\bigr)
        \;-\;k^{2}v^{h}_{\bm{m}},
   \qquad \bm{m}\in\Z^{d}.
   \label{eq:Lh}
\end{equation}

\paragraph{Protected region, lattice subsets, boundary strip}
Let $\Omegain\subset\R^{d}$ be a bounded Lipschitz domain with
$\Gammab=\partial\Omegain$, not assumed to align with $\Lambda_{h}$.
Define
\(
M^{+}=\{\bm{m}\in\Z^{d}:\,h\bm{m}\in\Omegain\},\;
M^{-}=\{\bm{m}\in\Z^{d}:\,h\bm{m}\notin\overline{\Omegain}\}.
\)
The five- or seven-point stencil sweep of $\Lh$ on $M^{\pm}$ produces
extended sets $N^{\pm}$, and the lattice boundary strip is
\(
\gamma=N^{+}\cap N^{-},
\)
split into $\gamma^{+}=\gamma\cap M^{+}$ and
$\gamma^{-}=\gamma\cap M^{-}$~\cite{Tsynkov2003,WangXia2025}. Functions
supported on $\gamma$ play the role of boundary-strip densities and
form the finite-dimensional trace space relevant to ANC.

\paragraph{Helmholtz LGF} The Helmholtz LGF $\Gh$ on $\Lambda_{h}$ is
the tempered fundamental solution of $\Lh$ selected by the Sommerfeld
radiation condition,
\begin{equation}
  \Gh_{\bm{m}}\;=\;\frac{1}{(2\pi)^{d}}\int_{[-\pi,\pi]^{d}}
       \frac{e^{i\,\bm{m}\cdot\bm{\theta}}}
            {\sigma(\bm{\theta})-k^{2}h^{2}+i0^{+}}\,d\bm{\theta},
  \qquad
  \sigma(\bm{\theta})\;=\;\sum_{j=1}^{d}2(1-\cos\theta_{j}).
  \label{eq:LGF}
\end{equation}
Efficient evaluation of $\Gh$ is provided in~\cite{WangXia2025}; for the
classical theory of discrete scattering and lattice Green's functions
see~\cite{BambergerGuillotJoly1988,Martin2006,BhatOsting2010}.
\begin{revision}
With the normalization in~\eqref{eq:LGF},
$\Lh\Gh=h^{-2}\delta_{\bm 0}$.  We therefore define the lattice
convolution by
\begin{equation}
  (\Gh\ast v^{h})_{\bm m}
  :=h^{2}\sum_{\bm n\in\Z^{d}}
       \Gh_{\bm m-\bm n}v^{h}_{\bm n}.
  \label{eq:convolution}
\end{equation}
For compactly supported lattice functions this convention gives
\begin{equation}
  \Lh(\Gh\ast v^{h})=v^{h},
  \label{eq:Ginv}
\end{equation}
and selects the outgoing solution.  The sum is finite for compactly
supported $v^{h}$.
\end{revision}

\paragraph{Nonresonance assumption} Throughout, we assume that
$k^{2}$ is not a Dirichlet eigenvalue of the discrete Laplacian on
$M^{+}$ with homogeneous data on $\gamma$.  Equivalently, the
interior lattice Helmholtz problem $L^{h}u^{h}=0$ on $M^{+}$ with
$u^{h}|_{\gamma}=0$ admits only the trivial solution.  This guarantees
unique solvability of the interior Dirichlet problem and well-posedness
of the Calder\'on projection.
\begin{newrevision}
In computation this condition is not binary in a robust sense: proximity
of $k^{2}$ to the discrete Dirichlet spectrum appears as a small singular
value of the relevant capacity or Dirichlet map and can substantially
amplify perturbations.  We therefore report condition numbers under grid
refinement in Section~\ref{sec:conditioning}; because both the difference
operator and the staircase geometry change with $h$, closeness to a
discrete resonance need not vary monotonically.
\end{newrevision}

\paragraph{Boundary-strip classification} Figure~\ref{fig:lattice}
illustrates the setup.  The sets $M^{\pm}$ and the
boundary strip $\gamma=\gamma^{+}\cup\gamma^{-}$ are computed from
a level-set function $\phi$ describing $\Omegain$ by a single 5-point
stencil sweep over the lattice.  No cut-cell interpolation,
body-fitting, or surface meshing is required; the protected region
geometry enters the method only through this classification step.

\section{LGF-based discrete Calder\'on projection}\label{sec:projection}

\begin{figure}[t]
\centering
\includegraphics[width=0.85\textwidth]{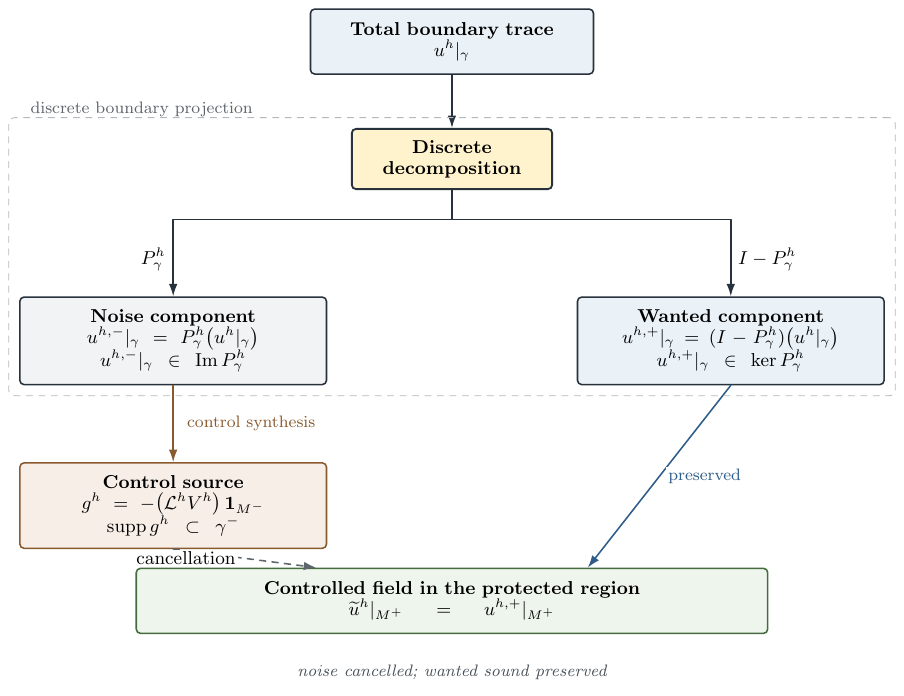}
\caption{Conceptual structure of the LGF Calder\'on projection for ANC
shielding.  The total-field trace on $\gamma$ is decomposed by $P^h_\gamma$
into a noise component $u^{h,-}|_\gamma\in\operatorname{Im}P^h_\gamma$ and
a wanted component $u^{h,+}|_\gamma\in\ker P^h_\gamma$.  The closed-form
control $g^h$ is supported on the single exterior layer $\gamma^-$ and
\rev{cancels the noise on $M^{+}$ while leaving the wanted field unaltered.}}
\label{fig:projection}
\end{figure}

\subsection{Definition}

Figure~\ref{fig:projection} sketches the conceptual structure of the
LGF projection and the resulting shielding control.

Let $\xi_{\gamma}:\gamma\to\C$ be a grid density.
\begin{revision}
We call $V^{h}:\Lambda_{h}\to\C$ an admissible extension of
$\xi_{\gamma}$ if $V^{h}|_{\gamma}=\xi_{\gamma}$,
$\Lh V^{h}$ has finite support, and either $V^{h}$ is compactly
supported or it is the outgoing solution generated by $\Lh V^{h}$.
For either class, uniqueness of the outgoing lattice problem gives
$\Gh\ast\Lh V^{h}=V^{h}$.  Compactly supported extensions always
exist; the outgoing class includes the single-layer realization used
in Theorem~\ref{prop:capacity}.
\end{revision}
For any admissible extension define
\begin{equation}
   \Ph\,\xi_{\gamma}
   \;:=\;
   V^{h}\;-\;\Gh\ast\bigl[(\Lh V^{h})\,\indicator{M^{+}}\bigr]
   \;=\;\Gh\ast\bigl[(\Lh V^{h})\,\indicator{M^{-}}\bigr],
   \label{eq:Pdef}
\end{equation}
and let
$\Ph_{\gamma}\xi_{\gamma}:=(\Ph\xi_{\gamma})|_{\gamma}$.

\begin{remark}[Comparison with~\cite{Tsynkov2003}]\label{rem:tsynkov}
The formula~\eqref{eq:Pdef} replaces the volumetric inverse
$(\Lh_{\Omega_{0}})^{-1}$ of~\cite{Tsynkov2003} by convolution with
$\Gh$. \begin{finalrevision}
Because $\Gh$ is finite at the origin and already satisfies the lattice
radiation condition, the projection requires neither singularity
subtraction nor a geometry-dependent artificial outer boundary.  The
finite tabulation used in computation approximates this infinite-lattice
kernel.
\end{finalrevision}
\end{remark}

\subsection{Main properties}

\begin{theorem}[LGF discrete Calder\'on
projection]\label{thm:projection}
Assume the nonresonance condition of Section~\ref{sec:prelim}
(the interior lattice Dirichlet problem on $M^{+}$ is uniquely
solvable). Then:
\begin{enumerate}[label=(\roman*)]
\item $\Ph_{\gamma}$ is independent of the chosen extension $V^{h}$.
\item $\Ph_{\gamma}$ is idempotent: $(\Ph_{\gamma})^{2}=\Ph_{\gamma}$.
\item $\xi_{\gamma}\in\operatorname{Im}\Ph_{\gamma}$ if and only if
there exists $u^{h}:N^{+}\to\C$ with $\Lh u^{h}=0$ on $M^{+}$ and
$u^{h}|_{\gamma}=\xi_{\gamma}$; in that case $\Ph\xi_{\gamma}|_{N^{+}}
=u^{h}$.
\end{enumerate}
\end{theorem}

\begin{revision}
\begin{proof}
Let $V^{h}_{1}$ and $V^{h}_{2}$ be admissible extensions of the same
trace and set $W^{h}=V^{h}_{1}-V^{h}_{2}$.  Write
$W^{h}=W^{h}_{+}+W^{h}_{-}$, where
$W^{h}_{\pm}=W^{h}\indicator{M^{\pm}}$.  Since
$W^{h}|_{\gamma}=0$, the stencil coupling across the boundary gives
\[
  (\Lh W^{h})\indicator{M^{+}}=\Lh W^{h}_{+}.
\]
Hence the difference between the two potentials defined by
\eqref{eq:Pdef} is
\[
  W^{h}-\Gh\ast\bigl[(\Lh W^{h})\indicator{M^{+}}\bigr]
  =W^{h}-\Gh\ast\Lh W^{h}_{+}=W^{h}_{-}.
\]
\newrev{The function $W^{h}_{-}$ vanishes on $N^{+}$: away from the
strip it is supported in $M^{-}$, while on the exterior strip nodes
$\gamma^{-}\subset N^{+}$ it vanishes because $W^{h}|_{\gamma}=0$.}  Thus
$\Ph\xi_{\gamma}|_{N^{+}}$, and in particular its trace on $\gamma$,
is independent of the extension.  This proves~(i).

For~(ii), let $u^{h}=\Ph\xi_{\gamma}|_{N^{+}}$ and
$\eta_{\gamma}=u^{h}|_{\gamma}=\Ph_{\gamma}\xi_{\gamma}$.
Applying $\Lh$ to~\eqref{eq:Pdef} shows that
$\Lh u^{h}=0$ on $M^{+}$.  Choose a compactly supported extension
$U^{h}$ of $u^{h}$ that agrees with it on $N^{+}$.  Then
$(\Lh U^{h})\indicator{M^{+}}=0$, and therefore
\[
  \Ph\eta_{\gamma}|_{N^{+}}=U^{h}|_{N^{+}}=u^{h}.
\]
Taking the trace on $\gamma$ yields
$\Ph_{\gamma}\eta_{\gamma}=\eta_{\gamma}$, which is
$(\Ph_{\gamma})^{2}=\Ph_{\gamma}$.

For~(iii), if $\xi_{\gamma}\in\operatorname{Im}\Ph_{\gamma}$,
write $\xi_{\gamma}=\Ph_{\gamma}\mu_{\gamma}$.  The function
$u^{h}=\Ph\mu_{\gamma}|_{N^{+}}$ is lattice Helmholtz on $M^{+}$ and
has trace $\xi_{\gamma}$.  Conversely, suppose that
$\Lh u^{h}=0$ on $M^{+}$ and $u^{h}|_{\gamma}=\xi_{\gamma}$.
Choose a compactly supported extension $U^{h}$ agreeing with $u^{h}$
on $N^{+}$.  Since
$(\Lh U^{h})\indicator{M^{+}}=0$, equation~\eqref{eq:Pdef} gives
$\Ph\xi_{\gamma}|_{N^{+}}=u^{h}$ and hence
$\Ph_{\gamma}\xi_{\gamma}=\xi_{\gamma}$.  Thus
$\xi_{\gamma}\in\operatorname{Im}\Ph_{\gamma}$.
\end{proof}
\end{revision}

\subsection{Comparison with the Tsynkov projection}

\begin{finalrevision}
The LGF and Tsynkov constructions identify the same canonical range:
the traces of interior lattice-Helmholtz solutions on $N^{+}$.  Their
complementary exterior trace spaces depend on the radiation treatment;
the LGF encodes the infinite-lattice outgoing condition, whereas a
bounded auxiliary problem encodes its chosen outer boundary operator.
\end{finalrevision}

\begin{theorem}[Range equivalence]\label{thm:equivalence}
Under the hypothesis of Theorem~\ref{thm:projection},
$\operatorname{Im}\Ph_{\gamma}$ $=$
$\operatorname{Im}\Ph_{\gamma,\Omega_{0}}$
for any Tsynkov-type discrete Calder\'on projection
$\Ph_{\gamma,\Omega_{0}}$ constructed from a well-posed
auxiliary problem on a bounded lattice-aligned domain
$\Omega_{0}\supset\overline{\Omegain}$.
Both projections therefore characterize the same space of
lattice-Helmholtz traces on $\gamma$.
\end{theorem}

\begin{proof}
By Theorem~\ref{thm:projection}(iii),
$\operatorname{Im}\Ph_{\gamma}$ consists of traces on $\gamma$ of
lattice-Helmholtz solutions $u^{h}$ satisfying $L^{h}u^{h}=0$ on
$M^{+}$.  The defining property of any Tsynkov-type projection
\cite{Tsynkov2003} implies the same characterization for
$\operatorname{Im}\Ph_{\gamma,\Omega_{0}}$.  The two sets are
therefore equal.
\end{proof}

\begin{revision}
\begin{remark}
If $\Omega_{0}$ is equipped with an exact discrete transparent
boundary condition that reproduces the lattice Green's function on
$\partial M_{0}$, then the two projections coincide as operators,
$\Ph_{\gamma,\Omega_{0}}=\Ph_{\gamma}$.  With an approximate outer
condition, Theorem~\ref{thm:equivalence} asserts equality of the ranges
only.  The projectors may have different kernels and may therefore act
differently on a general trace; equality of the resulting control
fields requires an additional consistency or error analysis.
\label{rem:tsynkov-range}
\end{remark}
\end{revision}

\section{Interior cancellation from the total-field trace}
\label{sec:controls}

\subsection{Setting}

Consider the continuous Helmholtz equation on $\R^{d}$ with
source $f=f^{+}+f^{-}$ supported respectively inside and outside
$\overline{\Omegain}$, and the outgoing Sommerfeld radiation
condition. Here $f^{+}$ is the \emph{wanted} interior sound and
$f^{-}$ is the \emph{adverse} exterior noise. Let $u^{\pm}$ be the
outgoing solutions of $Lu^{\pm}=f^{\pm}$; then $u=u^{+}+u^{-}$.

Discretize on $\Lambda_{h}$ with localized source vectors
$f^{h,\pm}$ on $M^{\pm}$. The discrete ANC problem asks for a control
$g^{h}$ with $\supp g^{h}\subseteq M^{-}$ such that the solution
$\widetilde{u}^{h}$ of
\begin{equation}
   \Lh\widetilde{u}^{h}\;=\;f^{h,+}+f^{h,-}+g^{h}\qquad\text{on }
   \Lambda_{h},
   \label{eq:Luh}
\end{equation}
with outgoing LGF tails, satisfies $\widetilde{u}^{h}|_{N^{+}}
=u^{h,+}|_{N^{+}}$, where $u^{h,+}=\Gh\ast f^{h,+}$.

\subsection{Kernel of the projection}

\begin{definition}[Exterior LGF projection]\label{def:Qh}
For $\xi_{\gamma}:\gamma\to\C$, pick any admissible extension
$V^{h}$ with $V^{h}|_{\gamma}=\xi_{\gamma}$ and set
$\Qh\xi_{\gamma}:=V^{h}-\Gh\ast[(\Lh V^{h})\indicator{M^{-}}]$,
$\Qh_{\gamma}\xi_{\gamma}:=(\Qh\xi_{\gamma})|_{\gamma}$.
\end{definition}

\begin{proposition}[Complementarity]\label{prop:Qh}
Under the hypothesis of Theorem~\ref{thm:projection} (and the outgoing
character of $\Gh$ for the exterior side), the operator $\Qh_{\gamma}$
is a projection onto the trace space of outgoing
lattice-Helmholtz solutions on $M^{-}\cup\gamma$, and
\begin{equation}
   \Ph_{\gamma}+\Qh_{\gamma}\;=\;\mathrm{I}_{\gamma}.
   \label{eq:complementary}
\end{equation}
\end{proposition}

The proof repeats Theorem~\ref{thm:projection} with $M^{+}$ replaced
by $M^{-}$; \eqref{eq:complementary} follows by adding~\eqref{eq:Pdef}
and the corresponding $\Qh$ formula and using $\Gh\ast\Lh V^{h}=V^{h}$.

\begin{newrevision}
The following theorem is the computational core of the method: after one
dense solve on the boundary strip, the two sublayer densities generate
the shielding and confinement components separately.
\end{newrevision}

\begin{revision}
\begin{theorem}[Capacity-matrix realization]\label{prop:capacity}
Let
\[
  (S_{\gamma\gamma})_{ij}
  =h^{2}\Gh(\gamma_i-\gamma_j),
  \qquad \gamma_i,\gamma_j\in\gamma,
\]
and assume that $S_{\gamma\gamma}$ is nonsingular.  For a trace
$\xi_{\gamma}$, let
$\lambda=S_{\gamma\gamma}^{-1}\xi_{\gamma}$ and extend
$\lambda$ by zero away from $\gamma$.  Set
$V^{h}=\Gh\ast\lambda$ and
$\lambda^{\pm}=\lambda\indicator{\gamma^{\pm}}$.
\newrev{This is the outgoing admissible extension from
Section~\ref{sec:projection}; compact support of $V^{h}$ is not
required.}  Then
$V^{h}|_{\gamma}=\xi_{\gamma}$, $\Lh V^{h}=\lambda$, and
\begin{equation}
  \Ph\xi_{\gamma}=\Gh\ast\lambda^{-},
  \qquad
  \Qh\xi_{\gamma}=\Gh\ast\lambda^{+}.
  \label{eq:capacity-projectors}
\end{equation}
Consequently, the single-layer shielding and confinement controls are
$g^{h}_{\rm sh}=-\lambda^{-}$ and
$g^{h}_{\rm conf}=-\lambda^{+}$, respectively.
\end{theorem}

\begin{proof}
The capacity equation gives $V^{h}|_{\gamma}=\xi_{\gamma}$, while
\eqref{eq:Ginv} gives $\Lh V^{h}=\lambda$.  Since
$\gamma^{+}\subset M^{+}$ and $\gamma^{-}\subset M^{-}$,
substitution in~\eqref{eq:Pdef} yields
\[
  \Ph\xi_{\gamma}
  =\Gh\ast\lambda
   -\Gh\ast(\lambda\indicator{M^{+}})
  =\Gh\ast\lambda^{-}.
\]
The identity for $\Qh$ follows analogously, and the control formulas
follow by changing the sign of the corresponding projected field.
\end{proof}
\end{revision}

\begin{proposition}\label{prop:kernel}
Under the hypothesis of Theorem~\ref{thm:projection},
$\ker\Ph_{\gamma}$ $=$ $\operatorname{Im}\Qh_{\gamma}$, and
$\ell^{2}(\gamma)=\operatorname{Im}\Ph_{\gamma}
\oplus\operatorname{Im}\Qh_{\gamma}$.
\end{proposition}

\begin{revision}
This is the discrete counterpart of the classical Calder\'on
decomposition into interior- and exterior-extendable traces.  It is an
immediate algebraic consequence of Proposition~\ref{prop:Qh}: since
$\Qh_{\gamma}=\mathrm{I}_{\gamma}-\Ph_{\gamma}$ and
$\Ph_{\gamma}$ is a projection,
\[
  \operatorname{Im}\Qh_{\gamma}=\ker\Ph_{\gamma},
  \qquad
  \ell^{2}(\gamma)=\operatorname{Im}\Ph_{\gamma}
  \oplus\operatorname{Im}\Qh_{\gamma}.
\]
\end{revision}

\subsection{Closed-form shielding control}

\begin{theorem}[Closed-form control]\label{thm:controls}
Under the hypothesis of Theorem~\ref{thm:projection}, given the
measured boundary-strip trace $\xi_{\gamma}=u^{h}|_{\gamma}$ of the
total field, pick any compactly supported $V^{h}$ with
$V^{h}|_{\gamma}=\xi_{\gamma}$. Define
\begin{equation}
   g^{h}\;=\;-(\Lh V^{h})\cdot\indicator{M^{-}}.
   \label{eq:control}
\end{equation}
Then the solution $\widetilde{u}^{h}$ of~\eqref{eq:Luh} satisfies
\begin{equation}
   \widetilde{u}^{h}\big|_{N^{+}}
   \;=\;u^{h,+}\big|_{N^{+}}.
   \label{eq:cancellation}
\end{equation}
Equivalently, the adverse component is cancelled inside the protected
region while the wanted component is preserved. Furthermore, $g^{h}$
is independent of $V^{h}$ up to an additive term annihilated on
$N^{+}$ by LGF convolution.
\end{theorem}

\begin{proof}
Decompose the total field on the lattice as $u^{h}=u^{h,+}+u^{h,-}$,
where $u^{h,\pm}$ is the outgoing LGF solution of $L^{h}u^{h,\pm}=f^{h,\pm}$;
in particular
\begin{equation}
  L^{h}u^{h,-}=0\quad\text{on }M^{+},
  \qquad
  L^{h}u^{h,+}=0\quad\text{on }M^{-}.
  \label{eq:pm-homog}
\end{equation}
Let $\xi_{\gamma}=u^{h}|_{\gamma}$, and pick any compactly supported
extension $V^{h}$ with $V^{h}|_{\gamma}=\xi_{\gamma}$.

From~\eqref{eq:control} and~\eqref{eq:Ginv}, the control response on the
lattice is
\begin{equation}
  w^{h}=\Gh\ast g^{h}
  =-\,\Gh\ast\bigl[(L^{h}V^{h})\,\indicator{M^{-}}\bigr].
  \label{eq:wh-v2}
\end{equation}
Since $\indicator{M^{+}}+\indicator{M^{-}}
=\indicator{\Lambda_{h}}$ and $\Gh\ast L^{h}V^{h}=V^{h}$
by~\eqref{eq:Ginv}, we rewrite the exterior convolution as
\begin{equation}
  \Gh\ast\bigl[(L^{h}V^{h})\,\indicator{M^{-}}\bigr]
  =V^{h}-\Gh\ast\bigl[(L^{h}V^{h})\,\indicator{M^{+}}\bigr].
  \label{eq:split-v2}
\end{equation}
Substituting into~\eqref{eq:wh-v2} and restricting to $N^{+}$,
\begin{equation}
  w^{h}\big|_{N^{+}}
  =\bigl(\Gh\ast\bigl[(L^{h}V^{h})\,\indicator{M^{+}}\bigr]
  -V^{h}\bigr)\Big|_{N^{+}}
  =-\,\Ph\,\xi_{\gamma}\big|_{N^{+}},
  \label{eq:w-equals-minus-P}
\end{equation}
by the definition~\eqref{eq:Pdef} of the discrete Calder\'on potential.

By linearity of $\Ph$,
\begin{equation}
  \Ph\,\xi_{\gamma}
  =\Ph\bigl(u^{h,+}|_{\gamma}\bigr)+\Ph\bigl(u^{h,-}|_{\gamma}\bigr).
  \label{eq:linearity}
\end{equation}
We evaluate the two terms separately on $N^{+}$.

\emph{(a) Noise component lies in the range.} By~\eqref{eq:pm-homog},
$u^{h,-}$ satisfies $L^{h}u^{h,-}=0$ on $M^{+}$ and has trace
$u^{h,-}|_{\gamma}$ on $\gamma$. By Theorem~\ref{thm:projection}(iii),
\begin{equation}
  \Ph\bigl(u^{h,-}|_{\gamma}\bigr)\big|_{N^{+}}=u^{h,-}\big|_{N^{+}}.
  \label{eq:range-noise}
\end{equation}

\emph{(b) Wanted component lies in the kernel.} By~\eqref{eq:pm-homog},
$u^{h,+}$ satisfies $L^{h}u^{h,+}=0$ on $M^{-}$, i.e.\ $u^{h,+}$ is an
outgoing lattice-Helmholtz solution on the \emph{exterior}
$M^{-}$. By Proposition~\ref{prop:kernel}, such a trace
lies in the kernel of $\Ph_{\gamma}$ on $N^{+}$:
\begin{equation}
  \Ph\bigl(u^{h,+}|_{\gamma}\bigr)\big|_{N^{+}}=0.
  \label{eq:kernel-wanted}
\end{equation}

Combining~\eqref{eq:linearity}, \eqref{eq:range-noise},
and~\eqref{eq:kernel-wanted},
\begin{equation}
  \Ph\,\xi_{\gamma}\big|_{N^{+}}=u^{h,-}\big|_{N^{+}}.
  \label{eq:P-equals-noise}
\end{equation}

Substituting~\eqref{eq:P-equals-noise} into~\eqref{eq:w-equals-minus-P},
\begin{equation}
  w^{h}\big|_{N^{+}}=-\,u^{h,-}\big|_{N^{+}}.
  \label{eq:w-cancels-noise}
\end{equation}
By linearity, $\widetilde{u}^{h}=u^{h,+}+u^{h,-}+w^{h}$, and therefore
\begin{equation}
  \widetilde{u}^{h}\big|_{N^{+}}
  =\bigl(u^{h,+}+u^{h,-}-u^{h,-}\bigr)\big|_{N^{+}}
  =u^{h,+}\big|_{N^{+}},
\end{equation}
which is~\eqref{eq:cancellation}.
\end{proof}

\begin{revision}
\begin{proposition}[Control support]\label{prop:support}
Assume that the capacity matrix $S_{\gamma\gamma}$ is nonsingular.
Then the shielding control can be chosen with
$\supp g^{h}\subseteq\gamma^{-}$, i.e.\ on the single exterior lattice
layer adjacent to the protected region.
\end{proposition}

\begin{proof}
Apply Theorem~\ref{prop:capacity}.  The admissible extension
$V^{h}=\Gh\ast\lambda$ satisfies $\Lh V^{h}=\lambda$ with
$\supp\lambda\subseteq\gamma$.  Therefore
\[
  g^{h}=-(\Lh V^{h})\indicator{M^{-}}
       =-\lambda\indicator{\gamma^{-}},
\]
which has the asserted support.
\end{proof}
\end{revision}

\begin{remark}
\begin{revision}
The crossing nodes of $\gamma^{-}$ constitute the admissible actuator
layer.  A particular density may vanish at some of these nodes, so
minimality is understood in terms of the available support rather than
nonzero amplitude at every node.
\end{revision}  The resulting control has the form
of a discrete single-layer potential on the exterior boundary strip.
\label{rem:single-layer}
\end{remark}

\begin{theorem}[Volumetric cancellation with wanted-sound
preservation]\label{thm:cancellation}
Under the hypotheses of Theorem~\ref{thm:controls},
\(
\rev{\widetilde{u}^{h}|_{M^{+}}=u^{h,+}|_{M^{+}},}
\)
and outside $\Omegain$ the controlled field is the sum of the original
total field and the outgoing LGF response of $g^{h}$. The wanted
interior sound is preserved exactly; the control is a physically
outgoing field generated by secondary sources on a single exterior
layer adjacent to $\Gammab$.
\end{theorem}

\begin{remark}[Measurement noise]\label{rem:robustness}
\begin{revision}
For the capacity realization,
$\delta\lambda=S_{\gamma\gamma}^{-1}\delta\xi_{\gamma}$.  Hence, in
the Euclidean norm,
\[
  \frac{\|\delta\lambda\|_{2}}{\|\lambda\|_{2}}
  \leq \kappa_{2}(S_{\gamma\gamma})
       \frac{\|\delta\xi_{\gamma}\|_{2}}
            {\|\xi_{\gamma}\|_{2}},
\]
whenever the denominators are nonzero.  This is a worst-case
finite-dimensional perturbation bound for the density.  \newrev{For the
field evaluated on a target set $X$, an additional operator factor
appears:
\[
 \|\delta u^{h}\|_{2,X}
 \leq \|h^{2}G^{h}(X,\gamma)\|_{2}
       \|S_{\gamma\gamma}^{-1}\|_{2}
       \|\delta\xi_{\gamma}\|_{2}.
\]
Thus the condition number alone does not determine the physical-field
error.}  No monotone asymptotic law in $h$ is assumed.  The experiments in Section~\ref{sec:numerics} show a
noise floor set jointly by discretisation error and matrix
conditioning, consistent with the observations in
\cite{LimEtAl2009AIAA,LimEtAl2011JASA}.
\end{revision}
\end{remark}

\section{Exterior cancellation from the total-field trace:
         acoustic confinement of interior noise}
\label{sec:confinement}

\subsection{Motivation and setting}

The shielding problem of Section~\ref{sec:controls} cancels exterior
noise inside $\Omegain$.  A dual problem, of equal practical relevance,
is \emph{acoustic confinement}: an adverse source lies \emph{inside}
$\Omegain$ (a loud machine, a speaker in a meeting pod), and one
seeks controls supported near the boundary from the interior side that
cancel the adverse field outside $\Omegain$ while preserving any
ambient exterior sound.  The LGF construction handles both sides
symmetrically because convolution with $\Gh$ respects the radiation
condition regardless of source location.

Let $f_{\mathrm{adv}}$ be the interior adverse source (to be confined)
and $f_{\mathrm{amb}}$ the exterior ambient source (to be preserved).
The confinement problem asks for a control $g$ with
$\supp g\subseteq\overline{\Omegain}$ such that the modified
field satisfies $\widetilde{u}=u_{\mathrm{amb}}$ outside
$\overline{\Omegain}$.

\subsection{Discrete confinement and dual projection}

On the lattice, seek $g^{h}$ with $\supp g^{h}\subseteq M^{+}$ such
that the solution $\widetilde{u}^{h}$ of
$L^{h}\widetilde{u}^{h}=f^{h}_{\mathrm{adv}}+f^{h}_{\mathrm{amb}}+g^{h}$
satisfies $\widetilde{u}^{h}|_{N^{-}}=u^{h}_{\mathrm{amb}}|_{N^{-}}$.
This is the exact dual of the shielding constraint on $N^{+}$.

Define the exterior LGF projection
$\Qh\xi_{\gamma}:=V^{h}-\Gh\ast[(L^{h}V^{h})\,\indicator{M^{-}}]$,
$\Qh_{\gamma}\xi_{\gamma}:=(\Qh\xi_{\gamma})|_{\gamma}$.
Under nonresonance, $\Qh_{\gamma}$ is a projection onto the trace
space of outgoing lattice-Helmholtz solutions on $M^{-}\cup\gamma$,
and $\Ph_{\gamma}+\Qh_{\gamma}=\mathrm{I}_{\gamma}$~\cite{LoncaricRyabenkiiTsynkov2001}.

\begin{theorem}[Closed-form confinement control]\label{thm:confinement}
Assume the hypothesis of Proposition~\ref{prop:Qh}.  Given the
measured boundary trace $\xi_{\gamma}=u^{h}|_{\gamma}$, pick any
compactly supported extension $V^{h}$ with
$V^{h}|_{\gamma}=\xi_{\gamma}$ and define
\begin{equation}
  g^{h}\;=\;-(\Lh V^{h})\cdot\indicator{M^{+}}.
  \label{eq:control-conf}
\end{equation}
Then $\widetilde{u}^{h}|_{M^{-}\cup\gamma}
=u^{h}_{\mathrm{amb}}|_{M^{-}\cup\gamma}$:
the adverse interior field is cancelled exactly outside $\Omegain$
while the ambient exterior field is preserved.
\end{theorem}

\begin{revision}
\begin{proof}
Let $w^{h}=\Gh\ast g^{h}$.  From the definition of $\Qh$ and
\eqref{eq:control-conf},
\[
  w^{h}=-\Gh\ast\bigl[(\Lh V^{h})\indicator{M^{+}}\bigr]
       =-\Qh\xi_{\gamma}.
\]
Write the measured trace as
$\xi_{\gamma}=u^{h}_{\mathrm{adv}}|_{\gamma}
+u^{h}_{\mathrm{amb}}|_{\gamma}$.  The field generated by the
interior adverse source is outgoing and lattice Helmholtz on $M^{-}$,
so its trace lies in $\operatorname{Im}\Qh_{\gamma}$ and
\[
  \Qh(u^{h}_{\mathrm{adv}}|_{\gamma})|_{N^{-}}
  =u^{h}_{\mathrm{adv}}|_{N^{-}}.
\]
The ambient exterior field is lattice Helmholtz on $M^{+}$, so its
trace lies in $\operatorname{Im}\Ph_{\gamma}=\ker\Qh_{\gamma}$.
Therefore
$\Qh\xi_{\gamma}|_{N^{-}}=u^{h}_{\mathrm{adv}}|_{N^{-}}$ and
$w^{h}|_{N^{-}}=-u^{h}_{\mathrm{adv}}|_{N^{-}}$.  It follows that
$\widetilde{u}^{h}|_{N^{-}}=u^{h}_{\mathrm{amb}}|_{N^{-}}$.
Under the capacity-matrix hypothesis of
Theorem~\ref{prop:capacity}, the control is
$g^{h}=-\lambda^{+}$ and is supported on $\gamma^{+}$.
\end{proof}
\end{revision}

\begin{remark}
The confinement control~\eqref{eq:control-conf} differs from the
shielding control~\eqref{eq:control} only in the indicator:
$\indicator{M^{+}}$ for confinement versus $\indicator{M^{-}}$ for
shielding.  The two controls are supported on opposite sides of the
boundary strip ($\gamma^{+}$ versus $\gamma^{-}$) and cancel the
adverse field on opposite sides of $\Omegain$.  \begin{revision}
The identity $\Ph_{\gamma}+\Qh_{\gamma}=\mathrm{I}_{\gamma}$ expresses
a decomposition of the trace space, not simultaneous satisfaction of
two preservation objectives.  Adding the shielding and confinement
densities gives $-\lambda$ in the capacity realization and generally
cancels the full represented trace; a combined design therefore
requires its own coupled objective.
\end{revision}
\end{remark}

\begin{revision}
\section{One-sided measurement for pure-noise shielding}\label{sec:onesided}

The full shielding construction uses the trace on
$\gamma=\gamma^{+}\cup\gamma^{-}$.  We now consider the restricted
setting in which no unknown wanted source is present and only the
exterior-layer trace is measured.  The result requires two additional
invertibility assumptions, which also expose the possible instability
of the reconstruction.

Let
\[
  S_{--}=h^{2}\Gh(\gamma^{-},\gamma^{-}),
  \qquad
  S_{+-}=h^{2}\Gh(\gamma^{+},\gamma^{-}),
\]
where the first argument denotes target nodes and the second source
nodes.  Assume that $S_{--}$ is nonsingular and that the interior
lattice Dirichlet problem on $M^{+}$ with prescribed values on
$\gamma^{-}$ is uniquely solvable.  Given
$\xi_{\gamma^{-}}$, define
\begin{equation}
  T^{-}_{h}:=S_{+-}S_{--}^{-1},
  \qquad
  \xi_{\gamma^{+}}=T^{-}_{h}\xi_{\gamma^{-}}.
  \label{eq:T-h}
\end{equation}
The reconstruction matrix has size
$|\gamma^{+}|\times|\gamma^{-}|$ and satisfies
\begin{equation}
  \|T^{-}_{h}\|_{2}
  =\|S_{+-}S_{--}^{-1}\|_{2}
  \leq
  \frac{\|S_{+-}\|_{2}}{\sigma_{\min}(S_{--})}.
  \label{eq:T-bound}
\end{equation}
Thus a small singular value of $S_{--}$ can strongly amplify
measurement or discretisation errors.  \newrev{The displayed estimate is
a worst-case submultiplicative bound; equality would additionally
require alignment of the relevant singular directions, so typical
amplification may be milder.}

\begin{theorem}[One-sided measurement for pure-noise shielding]
\label{thm:one-sided}
Assume the hypotheses above and suppose that $f^{h,+}=0$.  Let
$\xi_{\gamma^{-}}=u^{h,-}|_{\gamma^{-}}$ and set
\[
  \lambda_{-}=S_{--}^{-1}\xi_{\gamma^{-}},
  \qquad
  v^{h}=\Gh\ast\lambda_{-},
\]
where $\lambda_{-}$ is extended by zero away from $\gamma^{-}$.
Then
$v^{h}|_{N^{+}}=u^{h,-}|_{N^{+}}$, and the control
\begin{equation}
  g^{h}_{-}=-\lambda_{-}
  \quad\text{on }\gamma^{-},
  \qquad g^{h}_{-}=0\quad\text{elsewhere},
  \label{eq:control-onesided}
\end{equation}
produces $\widetilde{u}^{h}|_{N^{+}}=0$.
\end{theorem}

\begin{proof}
By construction, $v^{h}|_{\gamma^{-}}=\xi_{\gamma^{-}}$.  Both
$v^{h}$ and $u^{h,-}$ satisfy the homogeneous lattice Helmholtz
equation on $M^{+}$.  Their difference therefore solves the interior
homogeneous problem with zero data on $\gamma^{-}$, and uniqueness
gives $v^{h}=u^{h,-}$ on $N^{+}$.  The field generated by
$g^{h}_{-}$ is $-v^{h}$, so it cancels $u^{h,-}$ on $N^{+}$.
Equation~\eqref{eq:T-h} is obtained by evaluating $v^{h}$ on
$\gamma^{+}$.
\end{proof}

If an unknown wanted interior source is present, the measured value on
$\gamma^{-}$ is
$u^{h,-}|_{\gamma^{-}}+u^{h,+}|_{\gamma^{-}}$.  The second component
is not generally determined by a homogeneous interior extension from
$\gamma^{-}$, and applying $T^{-}_{h}$ to the total measurement need
not reproduce the true trace on $\gamma^{+}$.  Consequently,
one-sided data alone do not in general preserve an unknown wanted
field.  The full two-sided trace, or prior subtraction of a known
wanted field, is then required.

\paragraph{Numerical verification.}
Table~\ref{tab:onesided} reports a grid-refinement study for an
analytic point source on the circular region.  The one-sided residual
is larger than the two-sided residual and the observed rate increases
from approximately $1.4$ to $1.7$.  These tests include consistency
error because the analytic source does not satisfy the lattice
equation exactly. 

\begin{table}[t]
\centering
\caption{Convergence of one-sided ($\gamma^{-}$-only) vs.\ two-sided
shielding (circle, $k=5$, analytic point source).}\label{tab:onesided}
\begin{tabular}{cccccc}
\toprule
$n$ & $h$ & Two-sided $\varepsilon_{\mathrm{rel}}$ &
  One-sided $\varepsilon_{\mathrm{rel}}$ &
  $\varepsilon_{\mathrm{one}}/\varepsilon_{\mathrm{two}}$ & Rate \\
\midrule
31   & $1.34\!\times\!10^{-1}$ & $2.55\!\times\!10^{-2}$ & $5.04\!\times\!10^{-2}$ & 2.0 & -- \\
63   & $6.72\!\times\!10^{-2}$ & $8.20\!\times\!10^{-3}$ & $1.90\!\times\!10^{-2}$ & 2.3 & 1.40 \\
127  & $3.36\!\times\!10^{-2}$ & $2.16\!\times\!10^{-3}$ & $6.96\!\times\!10^{-3}$ & 3.2 & 1.45 \\
255  & $1.68\!\times\!10^{-2}$ & $5.41\!\times\!10^{-4}$ & $2.17\!\times\!10^{-3}$ & 4.0 & 1.68 \\
\bottomrule
\end{tabular}
\end{table}

\begin{remark}[Practical interpretation]
The reconstruction $T^{-}_{h}$ is relevant when microphones are
available only outside the protected region and the measured field is
pure adverse noise, or when a known wanted contribution has first been
subtracted.  It is a geometry- and frequency-dependent map and should
be monitored through~\eqref{eq:T-bound}; it is not a uniformly stable
replacement for two-sided measurement.

\emph{Connection to NASC.}
\begin{newrevision}
The nonlocal active sound control scheme of Utyuzhnikov, Hu, and
Zhou~\cite{Utyuzhnikov2017,ZhouUtyuzhnikov2020,
HuUtyuzhnikov2022,HuUtyuzhnikov2024} uses two geometrically separated
Huygens surfaces: an observation surface supplies field data and a
control surface carries equivalent secondary sources.  The present
boundary strip $\gamma=\gamma^{+}\cup\gamma^{-}$ is the lattice
counterpart of this two-surface structure, but the two sublayers are
one stencil apart and arise from the discrete operator rather than from
independently meshed physical surfaces.  The map $T^{-}_{h}$ transfers
measurements from $\gamma^{-}$ to the missing trace on $\gamma^{+}$,
after which the capacity solve produces the control density.  Under the
stated invertibility assumptions this map is exact for the discrete
pure-noise problem; its practical stability is governed by the singular
values in~\eqref{eq:T-bound}.  NASC permits more flexible physical
separation of sensing and actuation surfaces, whereas the present result
emphasizes an exact boundary-strip algebra on a Cartesian lattice.
\end{newrevision}
\end{remark}
\end{revision}

\section{Numerical experiments}\label{sec:numerics}

\finalrev{All numerical experiments are two-dimensional.}
We validate the control construction on three
geometries: (a) a circular protected region of radius $0.5$ centered
at the origin; (b) an L-shaped region (bounding square
$1.4\times 1.4$, arm thickness $0.3$); (c) a star-shaped region with
radial profile $r(\theta)=0.5+0.12\cos(5\theta)$. All experiments use
the same lattice: $n=2^{7}-1=127$ nodes per dimension over the
sampling window $[-2.15,2.15]^{2}$, with mesh spacing $h\approx0.0336$,
wavenumber $k=5$ (roughly $38$ points per wavelength). The LGF is
precomputed on $[0,n-1]^{2}$ by the DST-with-Hankel-correction method
of~\cite{WangXia2025}.

\paragraph{Incident fields.}  Two classes of incident field are used.
\textbf{LGF-constructed} point sources are built by placing a discrete
delta at the nearest lattice node and convolving with the precomputed
LGF $\Gh$: $u^{h}_{\mathrm{LGF}}=\Gh\ast\delta_{\bm{m}_{0}}$.  These
fields satisfy the discrete Helmholtz equation exactly, and the
shielding control cancels them to machine precision --- a verification
of the algebraic identity.  \textbf{Analytic} incident fields are
evaluated from their continuum formulas directly on the grid:
\[
u^{\mathrm{pw}}(\bm{x})=e^{ik\mathbf{d}\cdot\bm{x}}\quad \mbox{(plane wave)}
\] 
and 
\[
u^{\mathrm{ps}}(\bm{x})=\frac{i}{4}H_{0}^{(1)}(k|\bm{x}-
\bm{x}_{s}|)\quad \mbox{(point source, free-space Hankel function)}
\]
These do \emph{not} satisfy the discrete Helmholtz equation; the residual after
shielding is therefore controlled by the discretisation error of the
5-point Laplacian and the boundary-strip approximation.
\begin{newrevision}
Heuristically, the smooth-field consistency error of the five-point
operator is $O(h^{2})$ away from a source, while sampling the continuum
field on a staircase boundary strip introduces geometry- and
source-distance-dependent pre-asymptotic terms.  For a Hankel source
outside the protected region, the solution is smooth inside but its
higher derivatives grow as the source approaches the strip.  The
measured rates therefore need not be exactly two on the available grids,
even though the finite-difference consistency error is second order.
\end{newrevision}  For the circular region at $k=5$
on the $127\times127$ grid, the analytic point source and plane wave
give comparable residuals ($2.2\times10^{-3}$ and
$2.0\times10^{-3}$), consistent with a common discretisation floor.  Unless stated otherwise, the experiments in
Sections~\ref{sec:noise}--\ref{sec:wanted} use LGF-constructed sources
(self-consistency check), while
Sections~\ref{sec:num-confinement}--\ref{sec:onesided-num} use analytic sources
(discretisation-error test).  Source positions are given in
Table~\ref{tab:setup}.

\begin{table}[htbp]
\centering
\caption{Computational setup and source positions.}\label{tab:setup}
\begin{tabular}{lcccc}
\toprule
Geometry & $|M^{+}|$ & $|\gamma|\,(|\gamma^{-}|)$ &
$\bm{x}_{0}^{-}$ (noise) & $\bm{x}_{0}^{+}$ (wanted) \\
\midrule
Circle ($r=0.5$) & 697 & 172 (88)   & $(0.9,0)$    & $(0,0)$ \\
L-shape          & 830 & 322 (163)  & $(1.0,0.5)$  & $(-0.2,-0.55)$ \\
Star (5-fold)    & 773 & 212 (108)  & $(0.9,0)$    & $(0,0)$ \\
\bottomrule
\end{tabular}
\end{table}

\subsection{Noise shielding}\label{sec:noise}

We first verify the full-density shielding control with LGF-constructed
point sources, which satisfy the discrete Helmholtz equation exactly.
Figure~\ref{fig:summary_abs} shows the magnitude of the total field
before and after control for all three geometries.  Before control, the
point-source field radiates through the protected region, producing
complex interference patterns that are visible even in the L-shaped
and star-shaped regions.  After control, \finalrev{the residual is at machine precision for the
smooth circular, reentrant L-shaped, and non-convex star-shaped regions.}
Control nodes on $\gamma^{-}$ are shown as cyan dots; they form a
single-exterior-layer ring around each protected region.

\begin{figure}[htbp]
\centering
\includegraphics[width=0.8\textwidth]{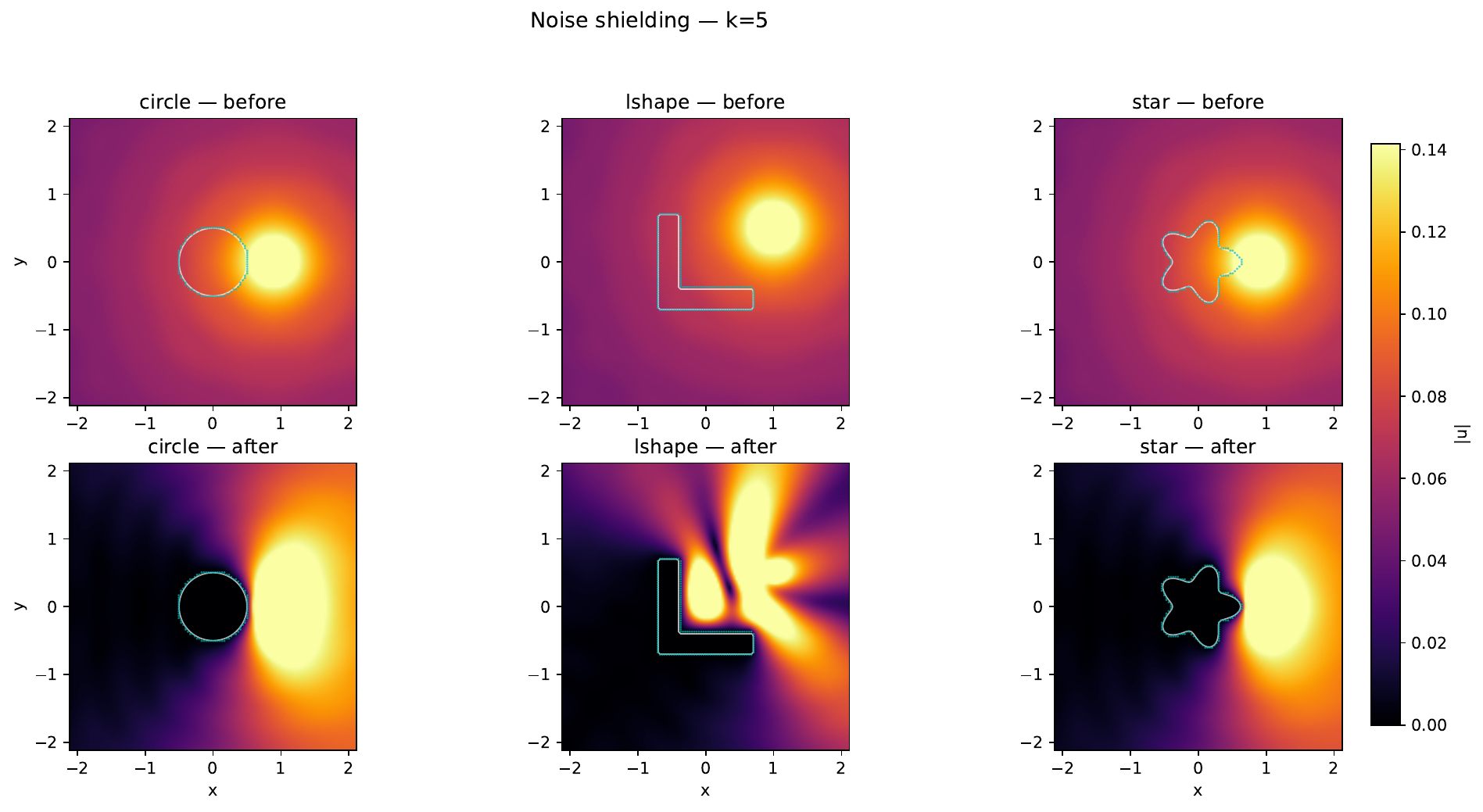}
\caption{Magnitude $|u^{h}|$ before (top row) and after (bottom row)
noise shielding for the three test geometries at $k=5$, LGF-constructed
point source.  Relative $\ell^{2}$ residuals: circle
$6.1\times10^{-15}$, L-shape $5.5\times10^{-15}$, star
$8.5\times10^{-15}$.}\label{fig:summary_abs}
\end{figure}

Figure~\ref{fig:attenuation} displays the pointwise attenuation
$20\log_{10}(|u^{h,-}|/|\widetilde{u}^{h}-u^{h,+}|)$ inside each
protected region on a dB scale.  Median attenuation exceeds $215$\,dB
for all three geometries, with the interior uniformly reaching
$>200$\,dB.  \finalrev{The attenuation remains high throughout the
protected region, showing volumetric cancellation rather than
interpolation only at the boundary-strip nodes.}

\begin{figure}[htbp]
\centering
\includegraphics[width=0.8\textwidth]{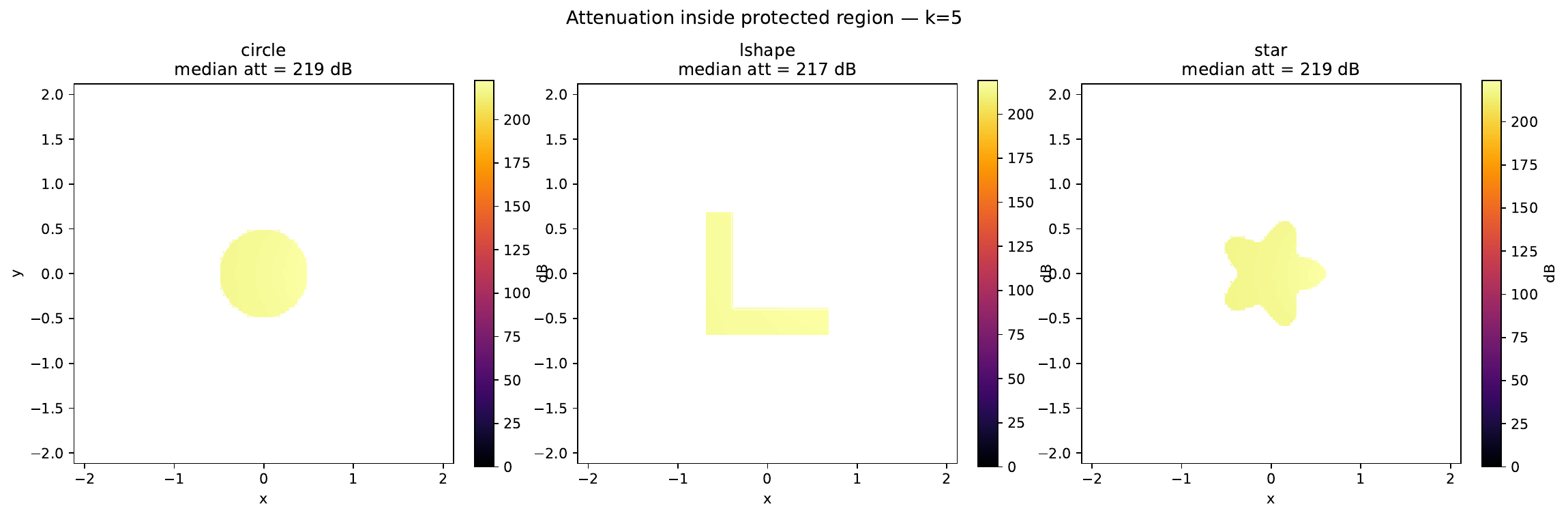}
\caption{Pointwise attenuation (dB) inside the protected region for the
three test geometries at $k=5$, LGF-constructed point source.  Median
attenuation: circle 219\,dB, L-shape 217\,dB, star 219\,dB.}\label{fig:attenuation}
\end{figure}

These results verify the discrete algebraic identity: when the incident
field is built from the same LGF used in the capacity matrix, the
control cancels it to machine precision.  The next section shows that
the cancellation is equally exact when an interior wanted source is
present.

\subsection{Wanted-sound preservation}\label{sec:wanted}

\begin{figure}[htbp]
\centering
\includegraphics[width=0.8\textwidth]{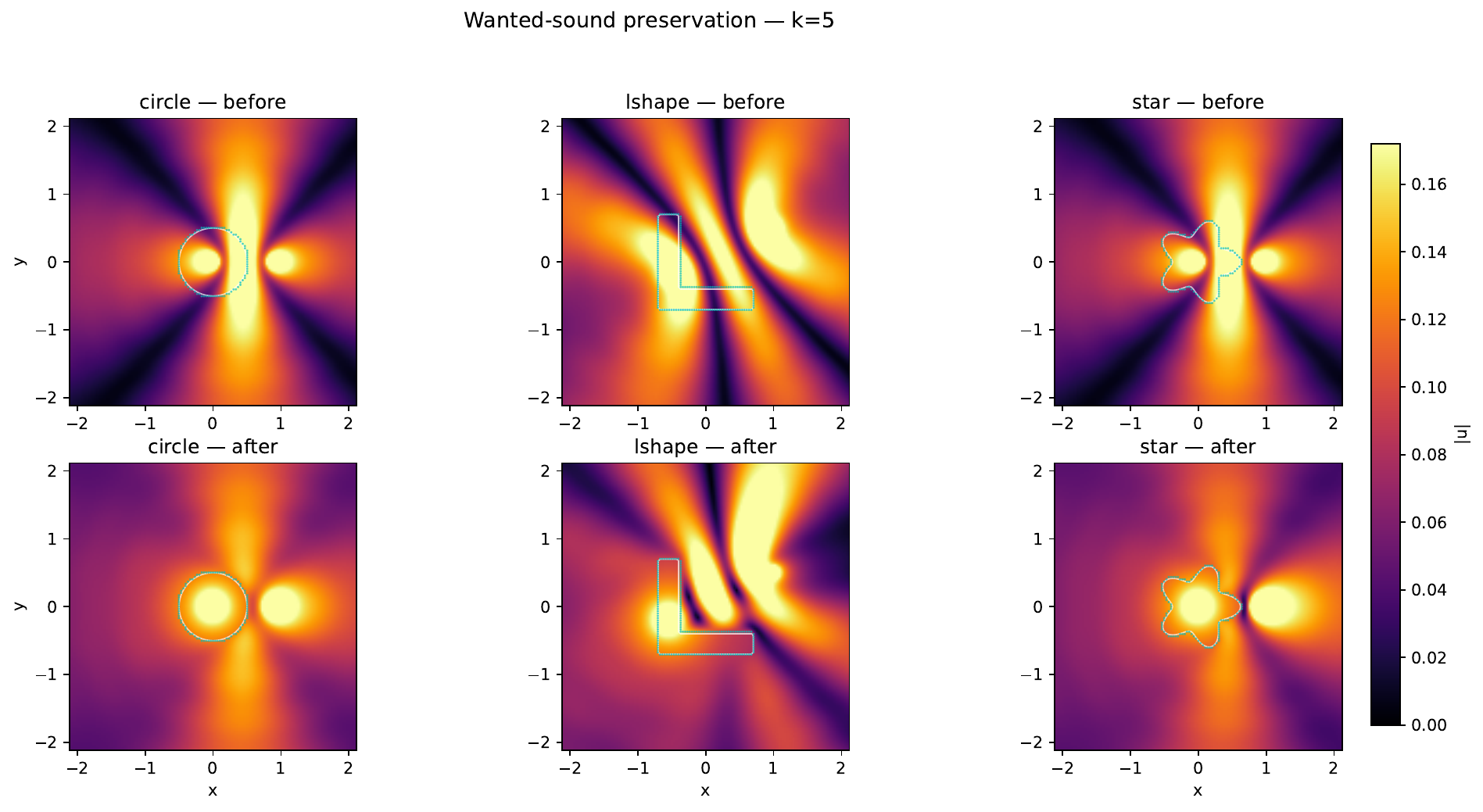}
\caption{Magnitude $|u^{h}|$ before (top row: noise $+$ wanted) and
after (bottom row: controlled) shielding for the three test geometries
at $k=5$, with interior wanted sources.}\label{fig:summary_abs_wanted}
\end{figure}

We add an interior wanted source to each geometry (positions in
Table~\ref{tab:setup}) and repeat the shielding computation.  The
control must now cancel the exterior noise while leaving the wanted
field unchanged inside $\Omegain$.  Figure~\ref{fig:summary_abs_wanted}
compares the total field before control (top row: noise $+$ wanted)
with the controlled field (bottom row).  After control, the interior
field equals $u^{h,+}$ to machine precision --- the wanted source is
preserved while the noise is eliminated.  This is the projection
property $\Ph_{\gamma}(u^{h,+}|_{\gamma})=0$ established in
Proposition~\ref{prop:kernel} in action: the wanted field lies in the
kernel of the Calder\'on projection, so it is annihilated by the
control formula.

\subsection{Acoustic confinement}\label{sec:num-confinement}

Figure~\ref{fig:confinement} demonstrates confinement on the circular
region ($R=0.5$) at $k=5$.  The adverse source is an interior
monopole at $(0,0)$ (amplitude $\times10$); the ambient field is a
plane wave $e^{ik\mathbf{d}\cdot\bm{x}}$,
$\mathbf{d}=(\sqrt{3}/2,0.5)$.  Control nodes are placed on
$\gamma^{+}$ (84 nodes, cyan dots).  After control, the exterior
field matches the plane wave to within a relative residual of
$1.4\times10^{-3}$, while the interior contains the confined adverse
source and the control response.

\begin{figure}[t]
\centering
\includegraphics[width=0.95\textwidth]{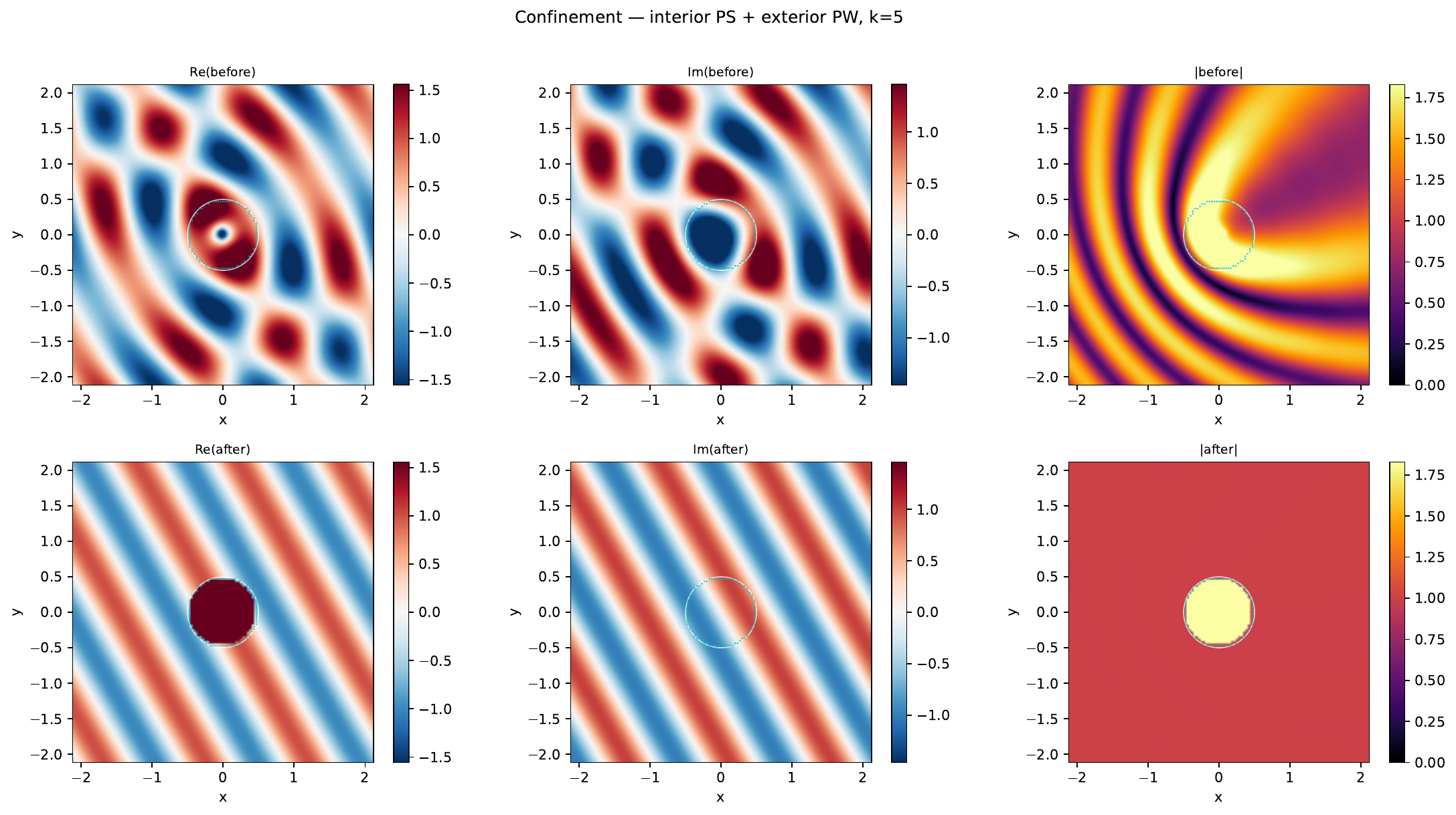}
\caption{Confinement: interior point source (adverse) confined by
controls on $\gamma^{+}$ (cyan dots, 84 nodes), exterior plane wave
(ambient) preserved.  Top row: before control.  Bottom row: after
control.  Columns: real part, imaginary part, magnitude.
The exterior field matches the ambient plane wave after control
(relative residual $1.4\times10^{-3}$, discretisation-limited).}
\label{fig:confinement}
\end{figure}

\subsection{Analytic incident fields}\label{sec:analytic}

Sections~\ref{sec:noise}--\ref{sec:wanted} used LGF-constructed point
sources, which satisfy the discrete Helmholtz equation exactly and are
cancelled to machine precision --- a verification of the algebraic
identity.  We now replace the LGF point source by two analytic incident
fields that do \emph{not} satisfy the discrete equation: a plane wave
$u^{\mathrm{pw}}=e^{ik\mathbf{d}\cdot\bm{x}}$
($\mathbf{d}=(\sqrt{3}/2,0.5)$) and a free-space point source
$u^{\mathrm{ps}}=\frac{i}{4}H_{0}^{(1)}(k|\bm{x}-\bm{x}_{s}|)$
with $\bm{x}_{s}=(0.9,0.0)$.  The LGF is used only for the capacity
matrix and control.

\begin{figure}[htbp]
\centering
\includegraphics[width=0.85\textwidth]{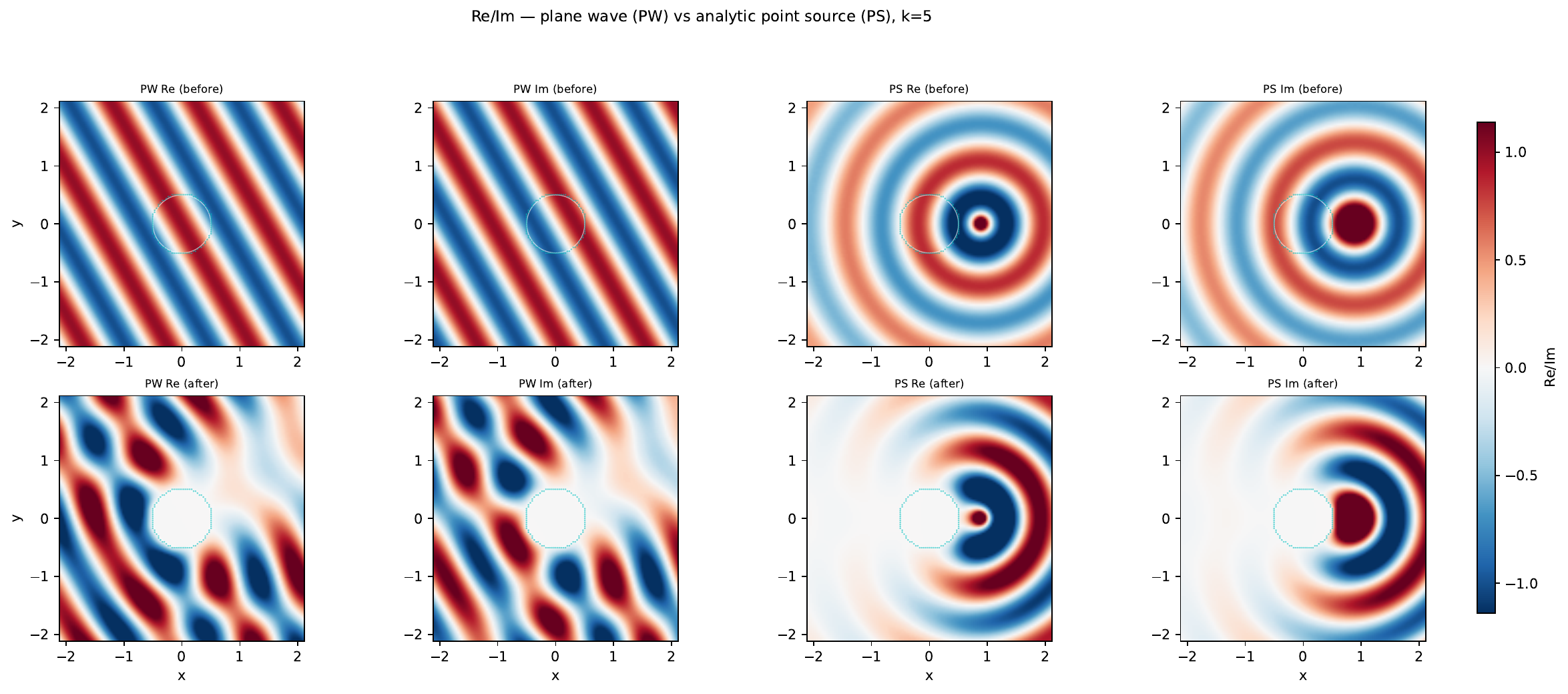}
\caption{Re/Im before (top row) and after (bottom row) shielding for
the circular region, $k=5$.  Columns 1--2: plane wave (PW).  Columns
3--4: analytic point source (PS).  The two residuals are comparable.}\label{fig:analytic_reim}
\end{figure}

Figure~\ref{fig:analytic_reim} compares the real and imaginary parts
before and after shielding for the circular region at $k=5$.  The
plane wave and point source produce nearly identical residuals
($2.0\times10^{-3}$ and $2.2\times10^{-3}$), consistent with the
same discretisation scale for the 5-point operator and boundary-strip
representation.  Combined $|\cdot|$ and Re/Im plots
for the L-shaped and star-shaped regions are provided as supplementary
material; the residuals are $1.3\times10^{-3}$ and $1.9\times10^{-3}$
(L-shape) and $1.9\times10^{-3}$ and $1.8\times10^{-3}$ (star) for
the plane wave and point source respectively.

\begin{figure}[htbp]
\centering
\includegraphics[width=0.55\textwidth]{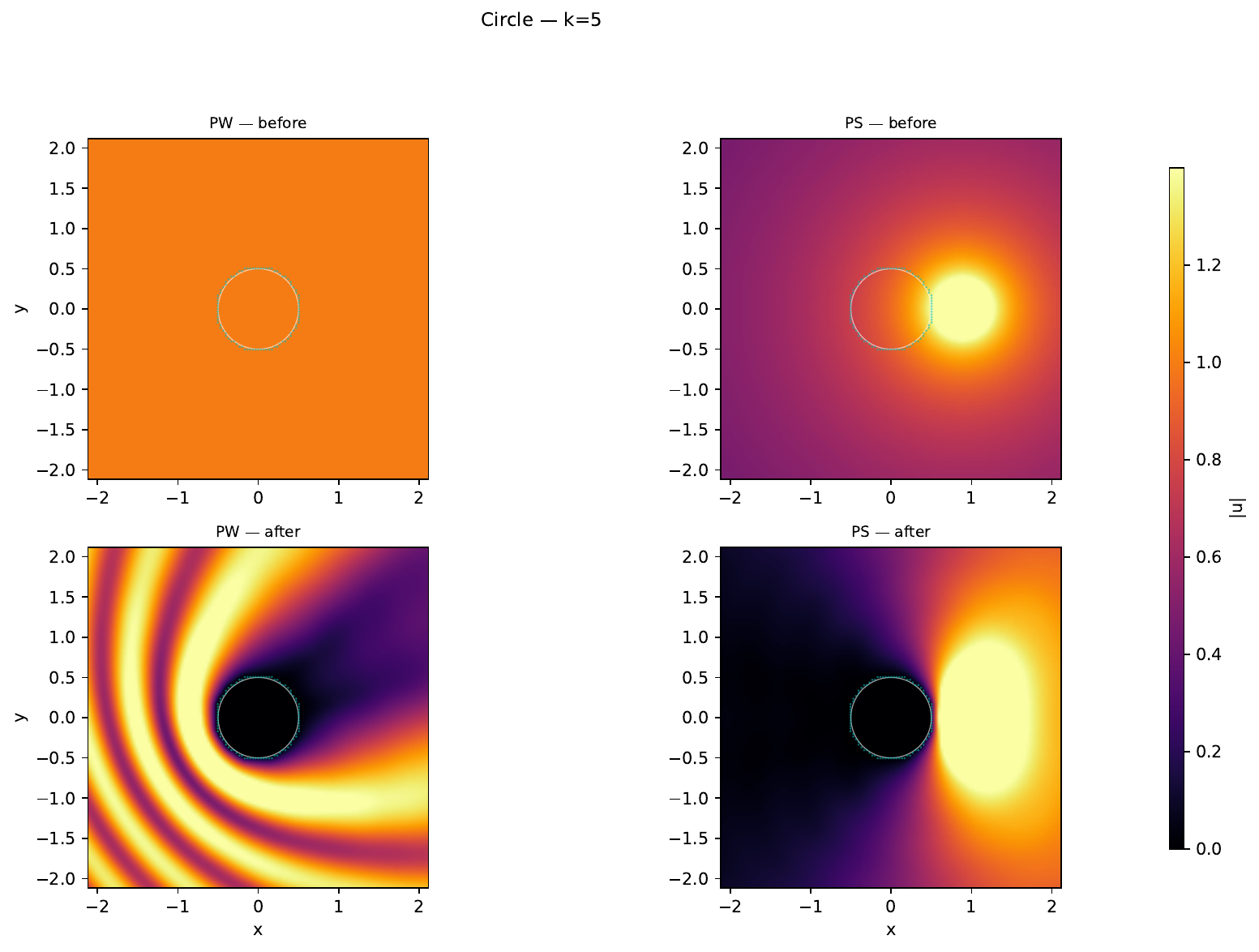}
\caption{Magnitude $|u^{h}|$ before (top) and after (bottom) shielding,
circular region.  Left: plane wave.  Right: analytic point source.}
\label{fig:analytic_abs}
\end{figure}

\begin{finalrevision}
Figures~\ref{fig:analytic_abs} and~\ref{fig:analytic_att} show comparable
controlled fields for the two incident-wave classes.  Across all three
geometries, the median attenuation is 57--60\,dB and the relative
residual is of order $2\times10^{-3}$.  These values are substantially
below the machine-precision attenuation obtained for LGF-consistent
sources because the analytic tests include the consistency error of the
discrete operator and boundary-strip representation.
\end{finalrevision}

\begin{figure}[htbp]
\centering
\includegraphics[width=0.75\textwidth]{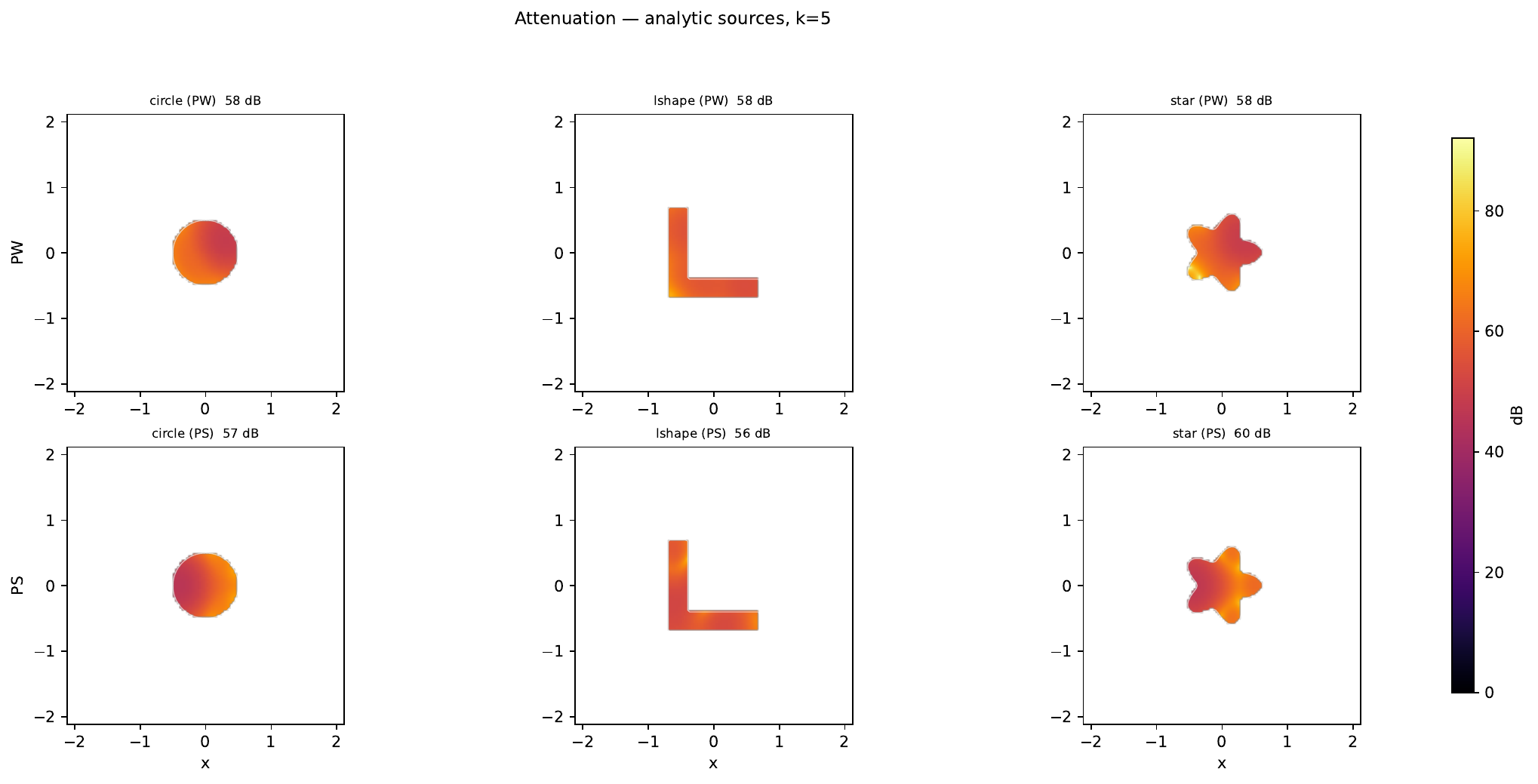}
\caption{Pointwise attenuation (dB) for plane wave (top row) and
analytic point source (bottom row), $k=5$.  Columns: circle (58/57\,dB
median PW/PS), L-shape (58/56\,dB), star (58/60\,dB).}\label{fig:analytic_att}
\end{figure}

We report the relative $\ell^{2}$ residual
$\varepsilon_{\mathrm{rel}}=\|\widetilde{u}^{h}-
u^{h,+}\|_{2,M^{+}}/\|u^{h,-}\|_{2,M^{+}}$
and the median pointwise attenuation inside $M^{+}$.

\rev{Both analytic fields are discretisation-limited and exhibit near-second-order convergence under grid refinement.}  Tables~\ref{tab:pwconv}
and~\ref{tab:psconv} report the plane-wave and analytic point-source
results for all three geometries at $n=31,63,127,255$ (pure noise).
\begin{revision}
The fitted rates are approximately $1.8$--$1.9$ over the four grids,
with the finest-grid pair approaching second order in several cases.
This is consistent with, but does not constitute a proof of, the
second-order truncation accuracy of the 5-point operator.
\newrev{At $n=31$ the diameter of the circular protected region is
resolved by fewer than eight mesh intervals, and the reentrant and
non-convex strips are similarly under-resolved; the first refinement
step should therefore be viewed as pre-asymptotic.}
\end{revision}  The first refinement step for the
circle and star (rate $\sim\!1.6$) is pre-asymptotic, due to the coarse
boundary representation on the $31\times31$ grid.  The attenuation
at $n=255$ reaches $\sim\!69$\,dB across all geometries and source
types, a factor of $10^{3.5}$ in amplitude reduction.

\begin{table}[htbp]
\centering
\caption{Plane-wave convergence: relative $\ell^{2}$ residual and
attenuation (dB) for all three geometries (pure noise, $k=5$).}\label{tab:pwconv}
\begin{tabular}{ccccccc}
\toprule
$n$ & \multicolumn{3}{c}{$\varepsilon_{\mathrm{rel}}$}
     & \multicolumn{3}{c}{Attenuation [dB]} \\
\cmidrule(lr){2-4}\cmidrule(lr){5-7}
 & Circle & L-shape & Star & Circle & L-shape & Star \\
\midrule
31   & $2.34\!\times\!10^{-2}$ & $2.11\!\times\!10^{-2}$ & $2.23\!\times\!10^{-2}$ & 37.9 & 34.1 & 37.3 \\
63   & $7.57\!\times\!10^{-3}$ & $5.54\!\times\!10^{-3}$ & $7.22\!\times\!10^{-3}$ & 46.4 & 45.5 & 46.4 \\
127  & $2.01\!\times\!10^{-3}$ & $1.34\!\times\!10^{-3}$ & $1.93\!\times\!10^{-3}$ & 57.8 & 57.6 & 57.4 \\
255  & $5.49\!\times\!10^{-4}$ & $3.68\!\times\!10^{-4}$ & $5.30\!\times\!10^{-4}$ & 69.0 & 68.9 & 68.7 \\
\midrule
Rate & $\approx\!1.8$ & $\approx\!1.9$ & $\approx\!1.8$ & & & \\
\bottomrule
\end{tabular}
\end{table}

\begin{table}[htbp]
\centering
\caption{Analytic point-source convergence (circle, L-shape, star;
$k=5$, amplitude $\times10$).  Residuals and attenuation are
comparable to the plane-wave case.}\label{tab:psconv}
\begin{tabular}{ccccccc}
\toprule
$n$ & \multicolumn{3}{c}{$\varepsilon_{\mathrm{rel}}$}
     & \multicolumn{3}{c}{Attenuation [dB]} \\
\cmidrule(lr){2-4}\cmidrule(lr){5-7}
 & Circle & L-shape & Star & Circle & L-shape & Star \\
\midrule
31   & $2.55\!\times\!10^{-2}$ & $2.97\!\times\!10^{-2}$ & $2.20\!\times\!10^{-2}$ & 38.6 & 30.1 & 38.3 \\
63   & $8.20\!\times\!10^{-3}$ & $8.94\!\times\!10^{-3}$ & $7.06\!\times\!10^{-3}$ & 46.1 & 40.3 & 48.8 \\
127  & $2.16\!\times\!10^{-3}$ & $2.18\!\times\!10^{-3}$ & $1.83\!\times\!10^{-3}$ & 56.9 & 52.7 & 59.8 \\
255  & $5.41\!\times\!10^{-4}$ & $5.45\!\times\!10^{-4}$ & $4.64\!\times\!10^{-4}$ & 69.1 & 64.8 & 71.5 \\
\midrule
Rate & $\approx\!1.8$ & $\approx\!1.9$ & $\approx\!1.9$ & & & \\
\bottomrule
\end{tabular}
\end{table}

\subsection{Two-sided vs.\ one-sided measurement}\label{sec:onesided-num}

\begin{table}[htbp]
\centering
\caption{Two-sided vs.\ one-sided ($\gamma^{-}$-only) shielding:
relative residual, attenuation, and wall-clock time
(circle, analytic point source, $k=5$).}\label{tab:onesided-comp}
\begin{tabular}{cccccccc}
\toprule
$n$  & $|\gamma|$ &
  $\varepsilon_{2}$ & Att$_2$ & $t_2$ (s) &
  $\varepsilon_{1}$ & Att$_1$ & $t_1$ (s) \\
\midrule
63  & 84  & $8.2\!\times\!10^{-3}$ & 46 & 0.13 & $1.9\!\times\!10^{-2}$ & 36 & 0.13 \\
127 & 172 & $2.2\!\times\!10^{-3}$ & 57 & 1.04 & $7.0\!\times\!10^{-3}$ & 45 & 1.04 \\
255 & 340 & $5.4\!\times\!10^{-4}$ & 69 & 8.32 & $2.2\!\times\!10^{-3}$ & 55 & 8.18 \\
\bottomrule
\end{tabular}
\end{table}

Table~\ref{tab:onesided-comp} compares full two-sided data with the
conditional one-sided reconstruction of
Theorem~\ref{thm:one-sided}.  Because the incident field is analytic
rather than LGF-constructed, both columns include discretisation
consistency error.  At $k=5$ the one-sided residual is a factor of
approximately $2$--$4$ larger and decreases under refinement.  

Figure~\ref{fig:onesided-comp} confirms that at $k=5$, $n=127$ the
one-sided shielding is visually indistinguishable from the two-sided
result.

\begin{figure}[htbp]
\centering
\includegraphics[width=0.65\textwidth]{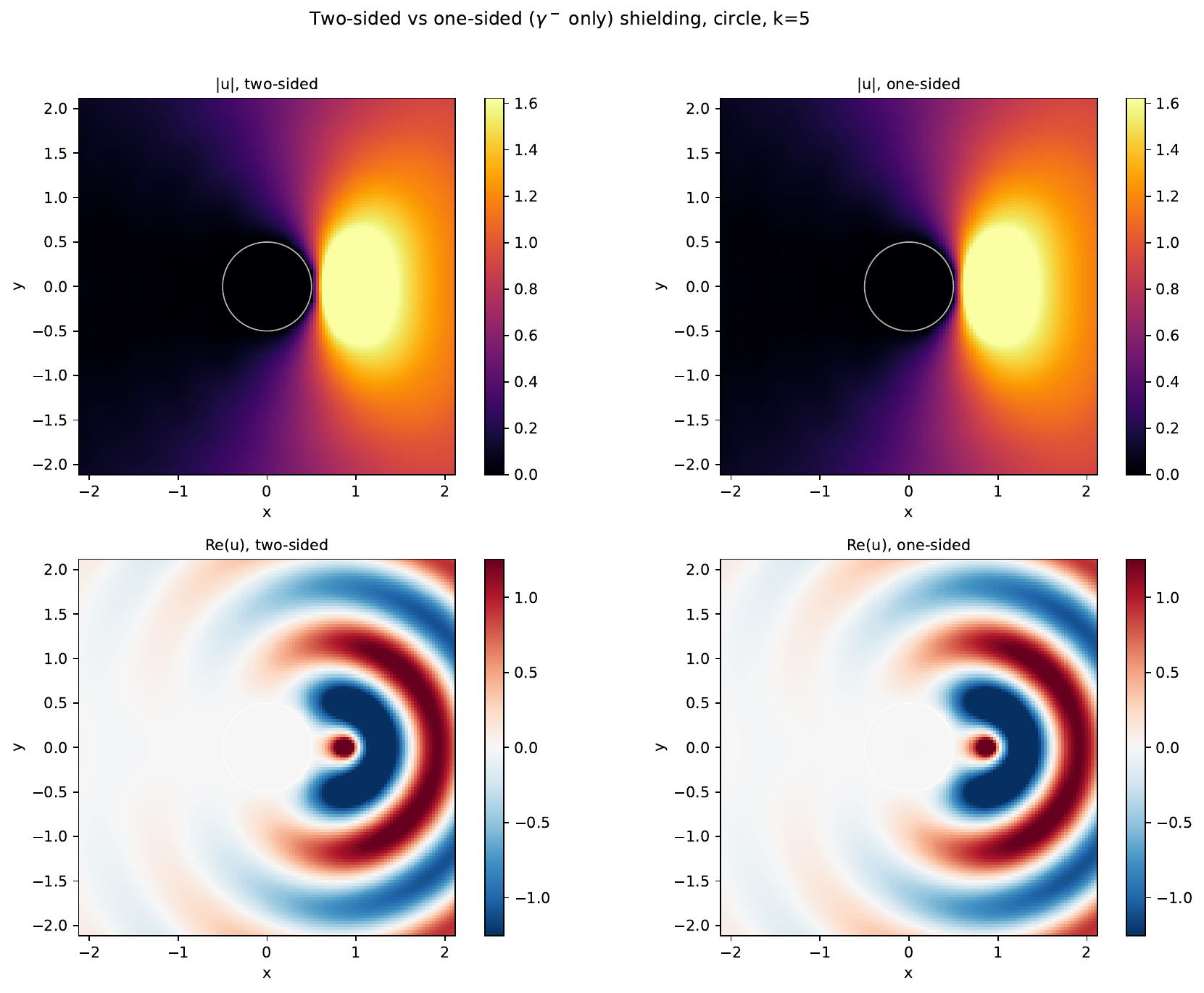}
\caption{Two-sided (left) vs.\ one-sided $\gamma^{-}$-only (right)
shielding, circle, $k=5$, $n=127$.  Top: magnitude.  Bottom: real
part.  \rev{At this configuration the two plots are visually close, although the one-sided residual is larger.}}\label{fig:onesided-comp}
\end{figure}

\subsection{Wavenumber and grid-refinement sensitivity}\label{sec:conditioning}

Table~\ref{tab:k10} extends the shielding results to $k=10$ across all
three test geometries using the analytic point source.  The attenuation
drops to 35--41\,dB (compared to $>215$\,dB with LGF-constructed
sources), consistent with the coarser resolution at higher wavenumber
($\sim\!19$ points per wavelength).  The condition numbers are
consistently lower at $k=10$ than at $k=5$.

\begin{table}[htbp]
\centering
\caption{Shielding results at $k=10$ across geometries
  ($127\times127$ grid, analytic point source).}\label{tab:k10}
\begin{tabular}{lcccc}
\toprule
Geometry & $|\gamma|$ & $\kappa(S_{\gamma\gamma})$ &
  $\varepsilon_{\mathrm{rel}}$ & Attenuation \\
\midrule
Circle ($r=0.5$) & 172 & $6.5\times10^{1}$  & $2.0\times10^{-2}$ & 36\,dB \\
L-shape          & 322 & $9.2\times10^{1}$  & $9.8\times10^{-3}$ & 41\,dB \\
Star (5-fold)    & 212 & $6.8\times10^{1}$  & $2.0\times10^{-2}$ & 35\,dB \\
\bottomrule
\end{tabular}
\end{table}

\begin{revision}
The smaller condition numbers at $k=10$ are specific to these grids and
geometries.  Changing $k$ shifts the discrete Helmholtz spectrum and may
move the problem either toward or away from a resonance; no monotone
conditioning trend with wavenumber is expected in general.
\end{revision}

\begin{finalrevision}
Table~\ref{tab:condh} reports the single-layer condition number under
grid refinement for the circular region at $k=5$.  The variation is
non-monotone: $\kappa(S_{\gamma\gamma})$ rises from $37.8$ at $p=6$ to
$361$ at $p=7$, then falls to $157$ at $p=8$.  Refinement changes both
the finite-difference spectrum and the staircase boundary, so the
discrete eigenvalue nearest $k^{2}$ may move toward or away from the
operating frequency.  The $p=7$ spike is therefore consistent with a
near-resonant configuration, although no specific eigenvalue distance
is computed here.

\begin{table}[t]
\centering
\caption{Single-layer capacity-matrix condition number under grid
refinement (circle, $k=5$).}\label{tab:condh}
\begin{tabular}{ccccc}
\toprule
$p$ & $n$ & $h$ & $|\gamma|$ & $\kappa(S_{\gamma\gamma})$ \\
\midrule
5 & 31  & $1.34\times10^{-1}$ & 44  & 14.1 \\
6 & 63  & $6.72\times10^{-2}$ & 84  & 37.8 \\
7 & 127 & $3.36\times10^{-2}$ & 172 & 361 \\
8 & 255 & $1.68\times10^{-2}$ & 340 & 157 \\
\bottomrule
\end{tabular}
\end{table}
\end{finalrevision}

\subsection{Measurement noise sensitivity}

In practice the boundary-strip trace $\xi_{\gamma}$ is obtained from
microphone measurements and is subject to noise.  To assess robustness,
we use the analytic point source (the more realistic test case) and add
independent circularly symmetric complex Gaussian noise to
$\xi_{\gamma}$ at relative standard deviations
$\sigma\in\{10^{-6},10^{-4},10^{-2},10^{-1}\}$, recompute the control
$g^h$ from the perturbed trace, and measure the resulting attenuation.
The results are summarised in Table~\ref{tab:noise}.

\begin{table}[t]
\centering
\caption{Effect of measurement noise on shielding performance
  (circle, $k=5$, analytic point source).  Attenuation in dB for one
  representative noise realization.}\label{tab:noise}
\begin{tabular}{ccc}
\toprule
Relative noise $\sigma$ & Attenuation (dB) & Degradation (dB) \\
\midrule
0 (clean)     & 56.9 & ---   \\
$10^{-6}$   & 56.9 & 0.0 \\
$10^{-4}$   & 56.8 & 0.1 \\
$10^{-2}$   & 34.4 & 22.5 \\
$10^{-1}$   & 12.1 & 44.8 \\
\bottomrule
\end{tabular}
\end{table}

With the analytic source, the clean attenuation is limited by
discretisation error (57\,dB) rather than by the algebraic identity,
so sub-discretisation-level measurement noise ($\sigma\leq10^{-4}$)
produces essentially no degradation — the method is robust up to
noise levels comparable to the discretisation floor.  At
$\sigma=10^{-2}$ (1\% relative noise), the method still retains
34\,dB of attenuation.  \begin{revision}
For larger $\sigma$, the degradation is consistent with amplification
by the inverse capacity matrix, as quantified by
Remark~\ref{rem:robustness}.  The table reports a representative
realization and therefore illustrates a trend rather than a statistical
confidence interval.  Regularization and ensemble noise studies are
left to future work.
\end{revision}

\section{Computational realization and complexity}\label{sec:algorithm}

\subsection{Algorithmic realization}\label{subsec:algorithm}

\begin{revision}
The implementation is the capacity-matrix realization proved in
Theorem~\ref{prop:capacity}.  Let
$S_{ij}=h^{2}\Gh(\gamma_i-\gamma_j)$ and solve
\begin{equation}
  S_{\gamma\gamma}\lambda=\xi_{\gamma}.
  \label{eq:capacity}
\end{equation}
For the particular outgoing extension
$V^{h}=\Gh\ast\lambda$, one has $\Lh V^{h}=\lambda$ and
$V^{h}|_{\gamma}=\xi_{\gamma}$.  Splitting
$\lambda=\lambda^{+}+\lambda^{-}$ by the two sublayers gives
\[
  \Ph\xi_{\gamma}=\Gh\ast\lambda^{-},
  \qquad
  \Qh\xi_{\gamma}=\Gh\ast\lambda^{+}.
\]
Thus the shielding control is
\begin{equation}
  g^{h}_{\rm sh}=-\lambda^{-},
  \label{eq:g-from-lambda}
\end{equation}
while the confinement control is $g^{h}_{\rm conf}=-\lambda^{+}$.
The corresponding field is evaluated by LGF convolution and added to
the measured or simulated total field.  The full potential
$\Gh\ast\lambda$ reproduces the original trace; it is the sublayer
potentials $\Gh\ast\lambda^{-}$ and $\Gh\ast\lambda^{+}$ that realize
the two Calder\'on components.
\end{revision}

\subsection{Comparison with truncated-domain realizations}

\begin{finalrevision}
For fixed $(h,k)$, the LGF table is reused across protected geometries
and source configurations.  A change of geometry requires a new
boundary-strip classification and capacity-matrix factorization, but no
volumetric Helmholtz solve on an auxiliary domain.  A truncated-domain
method may likewise reuse its factorization while the auxiliary domain
and outer treatment remain fixed; its projection, however, changes when
either choice changes.

The principal distinction is therefore the radiation model.  The LGF
selects the infinite-lattice outgoing solution directly, whereas a
bounded auxiliary realization inherits the accuracy of its outer
boundary operator.  Theorem~\ref{thm:equivalence} shows that both
constructions represent the same interior trace space; equality as
operators requires the auxiliary problem to reproduce the exact lattice
radiation condition.  Once assembled, the capacity matrix and LGF
evaluation maps are reusable for all source configurations associated
with the same boundary strip.
\end{finalrevision}

\subsection{LGF precomputation and complexity}

\begin{finalrevision}
The LGF depends only on $(h,k)$.  Tabulation on $[0,n-1]^2$ is therefore
performed once and reused across geometries; for each boundary strip,
assembling $S_{\gamma\gamma}$ requires $O(|\gamma|^{2})$ kernel lookups,
and a direct factorization costs $O(|\gamma|^{3})$ operations and
$O(|\gamma|^{2})$ storage.  The tested geometries have
$|\gamma|\approx100$--$300$, for which the direct implementation is
inexpensive and the factorization is reused across all incident fields.

A $d$-dimensional LGF table requires $O(n^{d})$ complex values.  This is
modest in the present two-dimensional experiments but becomes a primary
cost in three dimensions.  Scalable extensions will require symmetry,
on-demand kernel evaluation, compression, or hierarchical/FFT-based
application, together with iterative treatment of larger capacity
systems.
\end{finalrevision}

\section{Conclusions and outlook}\label{sec:outlook}

\begin{finalrevision}
We developed an infinite-lattice LGF realization of the discrete
Calder\'on projection for active acoustic shielding and confinement.
The outgoing radiation condition is built into the kernel, while the
geometry-specific computation reduces to a capacity system on the
lattice boundary strip.  The two sublayer densities yield exact
single-layer shielding and confinement controls, and exterior-only
measurements suffice for pure-noise shielding under explicit
invertibility conditions.

LGF-consistent sources verify the discrete identities to machine
precision on smooth, reentrant, and non-convex regions.  Analytic plane
waves and point sources converge at observed rates of approximately
$1.8$--$1.9$, and the conditioning and noise experiments identify the
main stability limitations.  The present direct implementation is
two-dimensional and best suited to moderate boundary-strip sizes.

The capacity-matrix formulation also provides a natural starting point 
for constrained and regularized active-control problems. In practical 
settings, the control amplitudes may be bounded, only a subset of boundary-strip
actuators may be available, and the measured trace may be incomplete 
or contaminated by noise. These situations lead to regularized least-squares, 
sparse-control, and sensor–actuator placement problems posed entirely on the 
boundary strip. The present work focuses instead on the canonical unconstrained 
projection in order to isolate the role of the infinite-lattice Green’s 
function and the exact discrete outgoing condition. Optimization under 
physical constraints will be considered separately.

\end{finalrevision}

\section*{Acknowledgments}
This work was partially funded by the Natural Science Foundation of China
(NSFC Grant No.~12401546) and Wenzhou Kean University
(Grant Nos.~ISRG2024003 and KY20250604000452). 

\bibliographystyle{unsrtnat}
\bibliography{cas-refs}

\end{document}